\documentclass[12pt]{article}

\usepackage{layout}
\usepackage[OT1]{fontenc}
\usepackage{color}
\usepackage{amsthm,amsmath,graphicx,latexsym,amssymb,amscd,amsfonts,enumerate, bm}
\usepackage[colorlinks,citecolor=blue,linkcolor=red,urlcolor=blue]{hyperref}

\numberwithin{equation}{section}
\numberwithin{figure}{section}
\theoremstyle{plain}
\newtheorem{theorem}{Theorem}
\newtheorem{proposition}{Proposition}
\newtheorem{lemma}{Lemma}
\numberwithin{theorem}{section}
\newtheorem{remark}{Remark}
\numberwithin{proposition}{section}
\numberwithin{lemma}{section}
\numberwithin{remark}{section}

\def\E{\mathbb{E}}
\def\R{\mathbb{R}}
\def\P{\mathbb{P}}
\def\D{\mathbb{D}}
\def\p{\partial}
\def\ud{\mathrm{d}}
\newcommand{\F}{\mathcal{F}}
\def\E{\mathbb{E}}
\def\R{\mathbb{R}}
\def\P{\mathbb{P}}
\def\ve{\varepsilon}
\def\l{\ell}
\def\H{\mathcal{H}}
\newcommand{\dd}{\mathrm{d}}

\begin{document}

\title{Strong approximation of time-changed stochastic differential equations involving drifts with random and non-random integrators\thanks{Kei Kobayashi's research was partially supported by a Faculty Fellowship at Fordham University.}
}
\date{}

\author{Sixian Jin\thanks{Fordham University. Email:\ sjin27@fordham.edu} \ and Kei Kobayashi\thanks{Corresponding author. Fordham University. Email:\ kkobayashi5@fordham.edu}}


\maketitle

\begin{abstract}
The rates of strong convergence for various approximation schemes are investigated for a class of stochastic differential equations (SDEs) which involve a random time change given by an inverse subordinator. SDEs to be considered are unique in two different aspects: i) they contain two drift terms, one driven by the random time change and the other driven by a regular, non-random time variable; ii) the standard Lipschitz assumption is replaced by that with a time-varying Lipschitz bound. The difficulty imposed by the first aspect is overcome via an approach that is significantly different from a well-known method based on the so-called duality principle.
On the other hand, the second aspect requires the establishment of a criterion for the existence of exponential moments of functions of the random time change.
\end{abstract}

\section{Introduction}\label{section_introduction}

Let $B=(B_t)_{t\ge0}$ be a standard Brownian motion and $E=(E_t)_{t\ge0}$ be a stochastic process defined by the inverse of a stable subordinator $D=(D_t)_{t\ge0}$ with index $\beta\in(0,1)$, independent of $B$. 
The composition process $B\circ E=(B_{E_t})_{t\ge0}$, called a \textit{time-changed Brownian motion}, and its various generalizations have been widely used to model ``subdiffusions'' arising in many different areas of science; 
see e.g.\ Chapter 1 of \cite{HKU-book}. 
There are a number of characteristics of the time-changed Brownian motion which make it fundamentally different from the regular Brownian motion. For example, $B\circ E$ is non-Markovian 
and 
its variance is $\E[B_{E_t}^2]=c t^\beta$ for some $c>0$, 
which shows that in large time scales
particles represented by $B\circ E$ diffuse more slowly than the regular Brownian particles. 
Moreover, each of these subdiffusive particles frequently gets trapped and become immobile for some time.
Further, the densities of $B\circ E$ satisfy the time-fractional heat 
  equation $\partial_t^\beta u(t,x) =(1/2)\Delta u(t,x)$, where $\partial_t^\beta$ denotes the Caputo fractional derivative of order $\beta$.
Various extensions of $B\circ E$ and their associated fractional order partial differential equations have been investigated, including time-changed fractional Brownian motions (see \cite{HKRU,HKU-2,MNX}).

This paper investigates 
stochastic differential equations (SDEs) of the form
\begin{align}\label{SDE-001-0}	
	\ud X_t=H(E_t,X_t)\, \ud t+ F(E_t,X_t)\, \ud E_t+G(E_t,X_t)\,\ud B_{E_t} 
	\, \ \textrm{with} \, \ X_0=x_0. 
\end{align}
One of the main difficulties of handing this SDE is the simultaneous presence of the two drift coefficients $H$ and $F$, with the former driven by the regular, non-random time variable and the latter driven by the random time change. 
Detailed analysis of the ``time-changed SDE'' \eqref{SDE-001-0} with L\'evy noise terms added first appeared in \cite{Kobayashi}, and 
since then, the time-changed SDE \eqref{SDE-001-0} and its extensions 
have drawn more and more attention. 
For example, Nane and Ni \cite{NaneNi2017,NaneNi2018} and Wu \cite{Wu}
established stability in various senses of solutions of time-changed SDEs; to overcome the difficulty of the simultaneous presence of the two drifts, they utilized extensions of the time-changed It\^o formula derived in \cite{Kobayashi}. 
On the other hand, the main contribution of this paper is to establish the rates of convergence for numerical approximation schemes for the time-changed SDE \eqref{SDE-001-0} with $H(E_t,X_t)=H(E_t)$, which is extremely important both theoretically and practically.

Let us suppose for the moment that $H\equiv 0$. Then the time-changed SDE can be effectively analyzed via the so-called \textit{duality principle} in Theorem 4.2 of \cite{Kobayashi}, which connects \eqref{SDE-001-0} with $H\equiv 0$ with the classical It\^o SDE 
\begin{align}\label{SDE-003-0}	
	\ud Y_t= F(t,Y_t)\, \ud t+ G(t,Y_t)\, \ud B_t \, \ \textrm{with} \, \ Y_0=x_0
\end{align}
in the following manner: if $Y_t$ solves \eqref{SDE-003-0}, then $X_t:=Y_{E_t}$ solves \eqref{SDE-001-0}, while if $X_t$ solves \eqref{SDE-001-0}, then $Y_t:=X_{D_t}$ solves \eqref{SDE-003-0}, where $D$ is the original subordinator. 
Indeed, Hahn et al.\ \cite{HKU-1} employed this one-to-one correspondence 
to establish a time-fractional pseudo-differential 
equation associated with a time-changed SDE involving jumps. 
The duality principle was also used 
in \cite{JumKobayashi} to discuss Euler--Maruyama-type approximation for SDE  \eqref{SDE-001-0} with $H\equiv 0$ under the standard Lipschitz assumption on the coefficients. The same approach was also employed in the more recent paper \cite{DengLiu2020} for semi-implicit Euler--Maruyama-type approximation for SDEs with superlinearly growing coefficients.

Let us now assume $H\not\equiv 0$. 
This paper derives extensions of the strong convergence result in \cite{JumKobayashi} while investigating the following non-trivial question: 
\begin{itemize}
\item[\textbf{(A)}] If $H\not\equiv 0$ and the coefficients $F$, $G$ and $H$ satisfy the Lipschitz assumption with a time-varying Lipschitz bound (as in Assumption $\mathcal{H}_1$ in Section \ref{section_3}), can we still obtain a convergence result similar to the one derived in \cite{JumKobayashi}?
\end{itemize}
Let us emphasize that \textit{the duality principle becomes powerless in the simultaneous presence of the two drifts $H$ and $F$, even with the standard Lipschitz assumption.}
This is because the corresponding classical SDE \eqref{SDE-003-0} would involve a term driven by $D$, and hence, $Y$ would depend on $E$, making the discussions provided in the above-cited papers \cite{DengLiu2020,HKU-1,JumKobayashi} no longer valid.

To overcome the difficulty, this paper utilizes a Gronwall-type inequality with a stochastic driver to control the moments of the error processes, which is \textit{intrinsically different from the approach taken in the above-mentioned papers \cite{NaneNi2017,NaneNi2018,Wu} in handling the simultaneous presence of the two drifts.}
Our method builds upon and extends the ideas presented in \cite{JinKobayashi} in dealing with a different type of time-changed SDEs of the form 
$
	\ud X_t=F(t,X_t)\, \ud E_t+ G(t,X_t)\, \ud B_{E_t}.
$
However, to answer Question \textbf{(A)} on the time-changed SDE \eqref{SDE-001-0}, we must appropriately modify that approach in order to deal with (i) the simultaneous presence of the two drifts $H$ and $F$ and (ii) the generalized Lipschitz assumption.   
To handle (ii), we derive a useful criterion for the existence of the exponential moments of functions of $E_t$ 
in Theorem \ref{expmoment}. The criterion generalizes Theorem 1 of \cite{JinKobayashi} and may be of independent interest to some readers.

We further investigate the following question: 
\begin{itemize}
\item[\textbf{(B)}] Can we improve the rate of convergence for the time-changed SDE \eqref{SDE-001-0} using an It\^o--Taylor-type approximation scheme? 
\end{itemize}
Even though we give a positive answer to this question only in the one-dimensional case and under the assumption that $H\equiv 0$, it is still a non-trivial generalization of the result established in \cite{JumKobayashi}. 
Indeed, Remark 3.2(4) of \cite{JumKobayashi} points out that a simple modification of their argument 
by using the duality principle together with the It\^o--Taylor scheme 
would \textit{not} lead to any improvement of the rate of convergence. 
It turns out that the approach taken for Question \textbf{(A)}, which completely avoids the use of the duality principle, enables us to tackle Question \textbf{(B)} effectively.

The rest of the paper is organized as follows. Section \ref{section_2} defines the wide class of random time changes to be considered in this paper and derives criteria for the existence and non-existence of various moments of the time changes. Section \ref{section_3} discusses the meaning of the time-changed SDE \eqref{SDE-001-0} and derives sufficient conditions for the existence of moments of the SDE solution. 
Using these results, 
Sections \ref{section_4} and \ref{section_5} establish the main theorems which answer Questions \textbf{(A)} and \textbf{(B)}. Section \ref{section_6} is devoted to numerical examples which verify the statements derived in Section \ref{section_4}.

\section{Inverse subordinators and their moments}\label{section_2}

Throughout the paper, 
$(\Omega,\F,\P)$ denotes a complete probability space, $\E$ denotes the expectation under $\P$, and all stochastic processes are defined on $(\Omega,\F,\P)$. 
Let $D=(D_t)_{t\ge 0}$ be a subordinator starting at $0$ with Laplace exponent $\psi$ with killing rate 0, drift 0, and L\'evy measure $\nu$; i.e.\ $D$ is a one-dimensional nondecreasing L\'evy process with c\`adl\`ag paths starting at 0 with Laplace transform 
$\mathbb{E}[e^{-sD_t}]=e^{-t\psi(s)}$,
where $\psi(s)=\int_0^\infty (1-e^{-sy})\,\nu(\ud y)$
with the condition
$\int_0^\infty (y\wedge 1)\, \nu(\ud y)<\infty$. 
\textit{We focus on the case when the L\'evy measure $\nu$ is infinite} (i.e.\ $\nu(0,\infty)=\infty$).
Let $E=(E_t)_{t\ge 0}$ be the \textit{inverse} (or \textit{hitting time process}) of $D$ defined by 
\[
	E_t:=\inf\{u>0; D_u>t\}, \ \ t\ge 0. 
\]
We call $E$ an \textit{inverse subordinator}. If the subordinator $D$ is stable with index $\beta\in(0,1)$, then $\psi(s)=s^\beta$ and $E$ is called an \textit{inverse $\beta$-stable subordinator.}   
The assumption that $\nu(0,\infty)=\infty$ implies that $D$ has strictly increasing paths with infinitely many jumps (see e.g.\ \cite{Sato}), and therefore, $E$ has continuous, nondecreasing paths starting at 0. Also, the inverse relation between $D$ and $E$ implies $\{E_t>x\}=\{D_x<t\}$ for all $t,x\ge 0$; see Figure \ref{figure_001a}. 
Note that the jumps of $D$ correspond to the (random) time intervals on which $E$ is constant, and during those constant periods, any time-changed process of the form  $X\circ E=(X_{E_t})_{t\ge 0}$ also remains constant. If $B$ is a standard Brownian motion independent of $D$, we can regard particles represented by the time-changed Brownian motion $B\circ E$ as being trapped and immobile during the constant periods; see Figure \ref{figure_001}.
 Note that even though $B\circ D$ is a L\'evy process, $B\circ E$ is not even a Markov process (see \cite{MS_1,MS_2}).

\begin{figure}
    \centering
\vspace{-120pt}
   \includegraphics[width=4.5in]{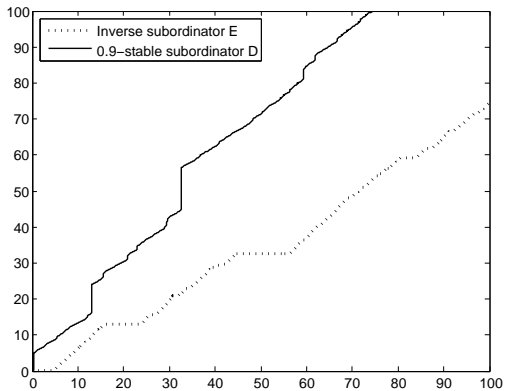}
\vspace{-120pt}
    \caption{Sample paths of a $0.9$-stable subordinator $D$ and the corresponding inverse $E$.}
    \label{figure_001a}
\end{figure}

\begin{figure}
    \centering
    \includegraphics[width=3.3in]{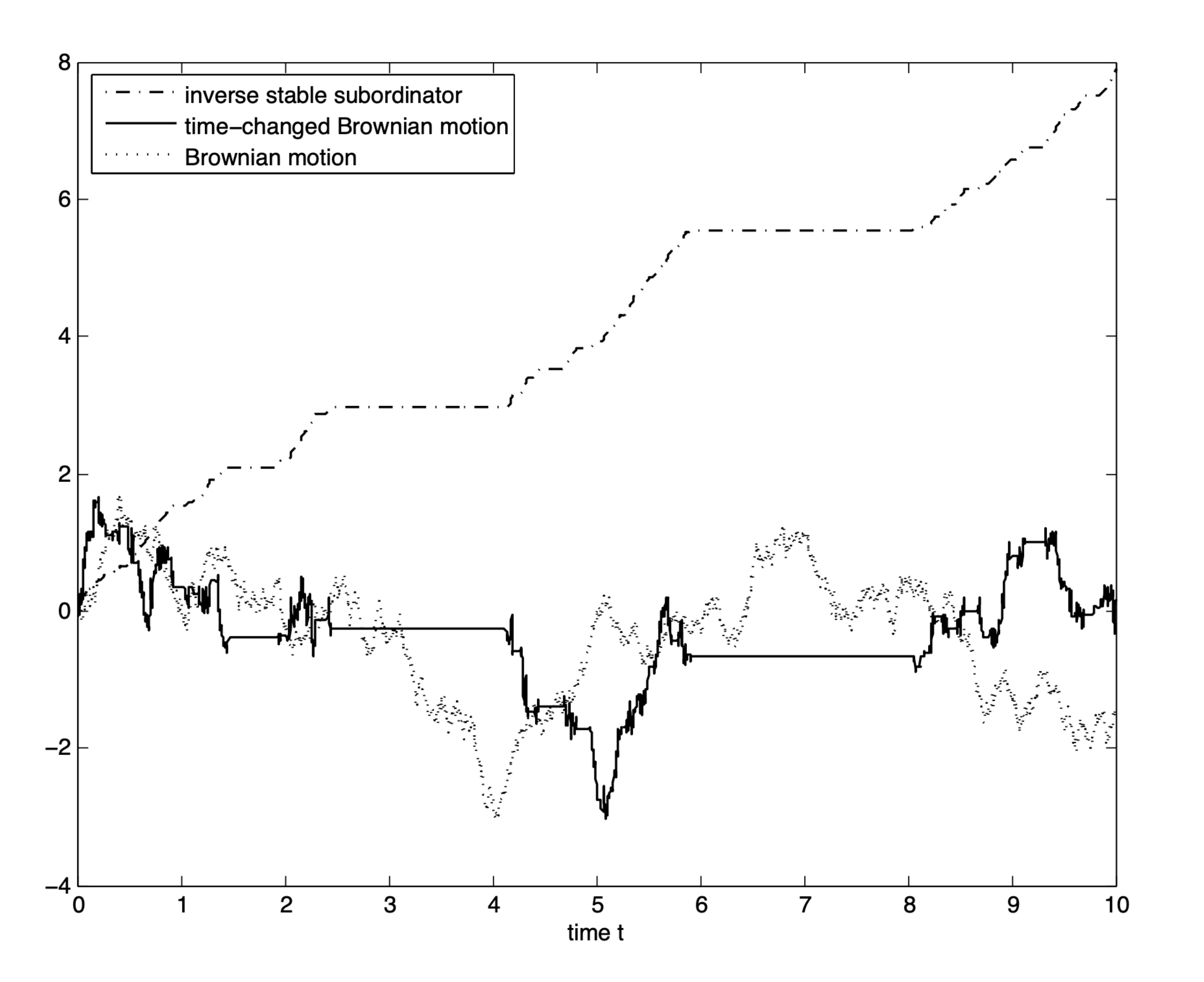}
    \caption{Sample paths of an inverse $0.8$-stable subordinator $E$ (dash-dotted) and the corresponding time-changed Brownian motion $B\circ E$ (solid), compared to the sample path of the underlying Brownian motion $B$ (dotted).}
    \label{figure_001}
\end{figure}

To describe the wide class of inverse subordinators $E$ to be discussed in this paper, let us introduce the notion of regularly varying and slowly varying functions. 
A function $f:(0,\infty)\to (0,\infty)$ is said to be \textit{regularly varying at $\infty$ with index $p\in\mathbb{R}$} if 
$
	\lim_{s\to \infty} f(cs)/f(s)=c^p
$
for any $c>0$. 
Let $\mathrm{RV}_p(\infty)$ denote the class of regularly varying functions at $\infty$ with index $p$.
A function $\ell:(0,\infty)\to (0,\infty)$ is said to be \textit{slowly varying at $\infty$} if $\ell\in \mathrm{RV}_0(\infty)$ 
(i.e.\ $\ell\in \mathrm{RV}_p(\infty)$ with $p=0$).
Every $f\in\mathrm{RV}_p(\infty)$ is represented as 
$
	f(s)=s^p \ell(s)
$
 with $\ell\in\mathrm{RV}_0(\infty)$.
Note that the following two Laplace exponents are regularly varying at $\infty$ with index $\beta\in(0,1)$: $\psi(s)=s^\beta$, which corresponds to a stable subordinator with index $\beta$, and $\psi(s)=(s+\kappa)^\beta-\kappa^\beta$ with $\kappa>0$, which corresponds to an exponentially tempered (or tilted) stable subordinator with index $\beta$ and tempering factor $\kappa$. 
On the other hand, $\psi(s)=\log(1+s)$, which corresponds to a Gamma subordinator, is slowly varying at $\infty$.

For the reader's convenience, we list some properties of regularly varying functions that are used throughout this paper (see Propositions 1.5.1 and 1.5.7 of \cite{Bingham_book}).
Let $f(0):=\lim_{s\to 0}f(s)$ and $f(\infty):=\lim_{s\to \infty}f(s)$ for a given function $f$ defined on $(0,\infty)$.

\begin{lemma}\label{lemma_RV} 
\begin{enumerate}[(i)] 
\item Given $f\in \mathrm{RV}_p(\infty)$, $f(\infty)=\infty$ if $p>0$, and $f(\infty)= 0$ if $p<0$. 
\item If $f_i\in \mathrm{RV}_{p_i}(\infty)$ for $i=1,2$ and $f_2(\infty)=\infty$, then $f_1\circ f_2\in \mathrm{RV}_{p_1p_2}(\infty)$.
\item If $f_i\in \mathrm{RV}_{p_i}(\infty)$ for $i=1,2$, then $f_1\cdot f_2\in \mathrm{RV}_{p_1+p_2}(\infty)$.  
\end{enumerate}
\end{lemma}

In Proposition \ref{proposition_RV} and Theorem \ref{expmoment} below, we provide important criteria for the existence and non-existence of various moments concerning inverse subordinators, which play key roles in the proofs of the statements to be established in Sections \ref{section_3}-\ref{section_5}.
Proposition \ref{proposition_RV} states that \textit{any} inverse subordinator with the underlying L\'evy measure being infinite has exponential moment; for proofs, see \cite{JumKobayashi,MagdziarzOrzelWeron}.

\begin{proposition}\label{proposition_RV}
Let $E$ be the inverse of a subordinator with infinite L\'evy measure. Then $\mathbb{E}[e^{\lambda E_t}]<\infty$ for any $\lambda>0$ and $t> 0$. 
Consequently, if $f:(0,\infty)\to(0,\infty)$ is a measurable function regularly varying at $\infty$ with index $p>0$ such that $\sup_{s\le s_0} f(s)<\infty$ 
for any $s_0<\infty$, then $\E[f(E_t)]<\infty$ for any $t>0$. 
\end{proposition}

Even though $E_t$ has exponential moment, its power $E_t^p$ with $p>1$ may or may not have exponential moment. For instance, if $E$ is an inverse $\beta$-stable subordinator, then 
 $\mathbb{E}[e^{\lambda E_t^2}]$ exists if $1/2<\beta<1$ while it does not if $0<\beta<1/2$. When $\beta=1/2$, whether the expectation exists or not depends on the relationship between $\lambda$ and $t$. 
For details, see Remark 6(2) of \cite{JinKobayashi}.
The following theorem generalizes Theorem 1 of \cite{JinKobayashi}.

\begin{theorem}
\label{expmoment}
Let $E$ be the inverse of a subordinator $D$ whose Laplace exponent $\psi$ is regularly varying at $\infty$ with index $\beta\in[0,1)$. 
If $\beta=0$, assume further that $\psi(\infty)=\infty$ and $\psi'$ is regularly varying at $\infty$ with index $-1$.
Suppose $f:(0,\infty)\to(0,\infty)$ is a measurable function regularly varying at $\infty$ with index $p>0$ such that $\sup_{s\le s_0} f(s)<\infty$ for any $0<s_0<\infty$.
Fix $t> 0$.
\begin{enumerate}[(i)]
\item If 
$p<1/(1-\beta)$, or equivalently, $\beta>(p-1)/p$, 
then $\mathbb{E}[e^{f(E_t)}]<\infty$. 
\item If 
$p>1/(1-\beta)$, or equivalently, $\beta<(p-1)/p$,
then $\mathbb{E}[e^{f(E_t)}]=\infty$. 
\end{enumerate}
\end{theorem}

\begin{proof} 
For any fixed $M>0$, $\int_0^M e^{f(u)}\,\P(E_t\in \ud u)\le e^{\sup_{u\le M} f(u)}<\infty$ by assumption. 
Hence, for large enough $M>0$, the integrability of $\mathbb{E}[e^{f(E_t)}]$ coincides with that of $\int_M^\infty e^{f(u)}\,\P(E_t\in\ud u)$.

If $p<1/(1-\beta)$, then there exist $q\in(p,1/(1-\beta))$ and $M_1>0$ such that $f(u)<u^q$ for all $u\ge M_1$. Indeed, for $q>p$, since $f(u)u^{-q}\in \mathrm{RV}_{p-q}(\infty)$, it follows from Lemma \ref{lemma_RV}(i) that $f(u)u^{-q}\to 0$ as $u\to\infty$.
Then $\int_{M_1}^\infty e^{f(u)}\,\P(E_t\in \ud u)\le \int_{M_1}^\infty e^{u^q}\,\P(E_t\in \ud u)\le \mathbb{E}[e^{E_t^q}]$, which is finite due to Theorem 1(1) of \cite{JinKobayashi}, so $\mathbb{E}[e^{f(E_t)}]<\infty$. On the other hand, if $p>1/(1-\beta)$, then again by Lemma \ref{lemma_RV}(i), there exist $q\in(1/(1-\beta),p)$ and $M_2>0$ such that $f(u)>u^q$ for all $u\ge M_2$. Since $\mathbb{E}[e^{E_t^q}]=\infty$ due to Theorem 1(2) of \cite{JinKobayashi}, it follows that $\int_{M_2}^\infty e^{f(u)}\,\P(E_t\in\ud u)\ge \int_{M_2}^\infty e^{u^q}\,\P(E_t\in\ud u)=\infty$. Thus, $\mathbb{E}[e^{f(E_t)}]=\infty$.
\end{proof}

Theorem \ref{theorem_negative} below concerning the moments of negative orders of $E_t$ (i.e.\ $\E[1/E_t^{p}]$ for $p>0$) can be regarded as a counterpart of Proposition \ref{proposition_RV} and Theorem \ref{expmoment}. 
To state the theorem in a more general setting, note that a function $f:(0,\infty)\to(0,\infty)$ is called \textit{regularly varying at 0 with index $p\in\R$}  if 
$
	\lim_{s\downarrow 0} f(cs)/f(s)=c^p
$
for any $c>0$, which is equivalent to the statement that $\tilde{f}\in\mathrm{RV}_{-p}(\infty)$ with $\tilde{f}(x):=f(1/x)$. 
The proof is based on the small ball probability of $E_t$ that is established in \cite{Kobayashi_smallball} or obtained immediately from a result in \cite{Rosinski_isomorphism}. 
Note that for any $t>0$ and $f:(0,\infty)\to(0,\infty)$, it follows that $f(E_t)$ is well-defined and positive a.s.\ since $E_t>0$ a.s.\ (or otherwise, the underlying subordinator $D$ would not start at 0).

\begin{theorem}\label{theorem_negative}
Let $E$ be the inverse of a subordinator with infinite L\'evy measure $\nu$. 
Suppose $f:(0,\infty)\to(0,\infty)$ is a measurable function regularly varying at 0 with index $p>0$ such that $\inf_{s\ge s_0} f(s)>0$ for any $0<s_0<\infty$.
Fix $t>0$.
\begin{enumerate}[(i)]
\item If $p<1$, then $\mathbb{E}[1/f(E_t)]<\infty$. 
\item If $p>1$ and $\nu[t,\infty)>0$, then
$\mathbb{E}[1/f(E_t)]=\infty$. \end{enumerate}
\end{theorem}
 
\begin{proof}
We first prove the statement in the special case when $f(u)=u^p$ for $u>0$. Observe that $\E\left[1/E_t^p\right]$ can be expressed as
\begin{align*}
	\int_0^\infty \P\left(\frac{1}{E_t^p}>x\right)\ud x
	=\int_0^\infty  \P(E_t<x^{-1/p})\ud x
	=\int_0^\infty p u^{-p-1}\P(E_t\le u)\ud u, 
\end{align*}
where we used the fact that the function $u\mapsto \P(E_t\le u)$ is a distribution function and hence $\P(E_t< u)=\P(E_t\le u)$ for a.e.\ $u$.
Since $\int_{u_0}^\infty p u^{-p-1}\P(E_t\le u)\,\ud u\le \int_{u_0}^\infty p u^{-p-1}\,\ud u=u_0^{-p}<\infty$ for any fixed $u_0>0$, it suffices to discuss the convergence and divergence of the integral $\int_0^{u_0} p u^{-p-1}\P(E\le u)\,\ud u$ for small enough $u_0>0$. On the other hand, by Corollary 4.14 of \cite{Rosinski_isomorphism} (also see Proposition 1 of \cite{Kobayashi_smallball}), 
\[
	\P(E_t\le u)\sim \nu[t,\infty)u \ \ \textrm{as} \ \ u\downarrow 0. 
\]
(If $\nu[t,\infty)=0$, this is interpreted as $\P(E_t\le u)=o(u)$.) 
Thus, if $0<p<1$, then   
$\int_0^{u_0} p u^{-p-1}\P(E_t\le u)\,\ud u\le  (\nu[t,\infty)+\ve)\int_{0}^{u_0} p u^{-p} \,\ud u<\infty$ for some $\ve>0$ and $u_0>0$, thereby yielding (i). In contrast, if $p\ge 1$ and $\nu[t,\infty)>0$, then for some $\ve\in(0,\nu[t,\infty))$ and $u_0>0$, $\int_0^{u_0} p u^{-p-1}\P(E_t\le u)\ud u\ge (\nu[t,\infty)-\ve)\int_{0}^{u_0} p u^{-p} \ud u=\infty$, from which (ii) follows.

Now, consider a general regularly varying function $f$ stated in the theorem. For any fixed $u_1>0$, 
$\int_{u_1}^\infty(1/f(u))\P(E_t\in\ud u)\le 1/[\inf_{u\ge u_1} f(u)]<\infty$ by assumption. Hence, the integrability of $\E[1/f(E_t)]$ coincides with that of $\int_0^{u_1}(1/f(u))\P(E_t\in \ud u)$ for small enough $u_1>0$.
Since $\tilde{f}(x)\in \mathrm{RV}_{-p}(\infty)$ with $\tilde{f}(x):=f(1/x)$, by Lemma \ref{lemma_RV}(i), $\lim_{u\downarrow 0}f(u)u^{-\alpha}= \infty$ if $\alpha>p$. Thus, if $0<p<1$, then there exist $\alpha\in(p,1)$ and $u_1>0$ such that $f(u)>u^\alpha$ for $u\le u_1$. For this $\alpha$, by the special case above, 
$
	\infty>\E[1/E_t^\alpha]\ge \int_0^{u_1}(1/{u^\alpha})\P(E_t\in\ud u)\ge \int_0^{u_1}(1/{f(u)})\P(E_t\in\ud u), 
$
which yields (i). 

On the other hand, if $p>1$, then since $\lim_{u\downarrow 0}f(u)u^{-\alpha}=0$ if $\alpha<p$, there exist $\alpha\in(1,p)$ and $u_2>0$ such that $f(u)<u^\alpha$ for $u\le u_2$. For this $\alpha$, by the special case above, $\E[1/E_t^\alpha]=\infty$, and hence, 
$\infty=\int_0^{u_2}(1/u^\alpha)\P(E_t\in\ud u)
	\le 	\int_0^{u_2}(1/{f(u)})\P(E_t\in\ud u),$
which yields (ii). 
\end{proof}

\begin{remark}
(1) The proof of Theorem \ref{theorem_negative}(ii) implies that, in the case when $p=1$, if $\nu[t,\infty)>0$ and there exists $c>0$ such that $f(s)\le cs$ for all $s$ small enough, then $\mathbb{E}[1/f(E_t)]=\infty$. 

(2) A simple application of Theorem \ref{theorem_negative} provides an insight into the one-sided exit problem for a time-changed Brownian motion $B\circ E$, where $B$ is a Brownian motion independent of the inverse subordinator $E$. By (8.3) in Chapter 2 of \cite{KaratzasShreve}, the running maximum $M_T:=\max_{0\leq t\leq T}B_t$ of the Brownian motion has density $f_{M_T}(x)=\sqrt{2/(\pi T)}\exp{(-x^2/2T)}$ with $x>0$, so 
$
	\sqrt{2/(\pi T)}e^{-\frac{1}{2T}}
	\le \P(\max_{0\le t\le T}B_t\le 1)
	\le \sqrt{2/(\pi T)}
$ 
for any fixed $T>0$.
Since $\P(\max_{0\le t\le T}B_{E_t}\le 1)=\P(\max_{0\le r\le E_T}B_r\le 1)$, a simple conditioning yields 
\begin{align*}
	\sqrt{\frac{2}{\pi}}\E\left[E_T^{-\frac12}e^{-\frac{1}{2E_T}}\right]
	\le \P(\max_{0\le t\le T}B_{E_t}\le 1)
	\le \sqrt{\frac{2}{\pi}}\E\left[E_T^{-\frac12}\right],
\end{align*}
where the finiteness of the lower and upper bounds is guaranteed by
Theorem \ref{theorem_negative}(i). 
\end{remark}

\section{Stochastic differential equations involving a random time change and associated $L^p$ bounds}\label{section_3}

Throughout the rest of the paper, let $(\F_t)_{t\ge 0}$ be a filtration on the probability space $(\Omega,\F,\P)$ satisfying the usual conditions and $B$ be an $m$-dimensional $(\F_t)$-adapted Brownian motion which is independent of an $(\F_t)$-adapted subordinator $D$ with infinite L\'evy measure.
Since $B$ and $D$ are assumed independent, it is possible to set them up on a product probability space 
in the following manner. 
Let $B=B(\omega_1)$ be a Brownian motion defined on a probability space $(\Omega_B,\mathcal{F}_B,\P_B)$ and $D=D(\omega_2)$ be a subordinator defined on another probability space $(\Omega_D,\mathcal{F}_D,\P_D)$. Consider $B$ and $D$ to be defined on the measurable space $(\Omega_B\times\Omega_D, \mathcal{F}_B\times \mathcal{F}_D)$ by simply setting $B(\omega_1,\omega_2):=B(\omega_1)$ and  $D(\omega_1,\omega_2):=D(\omega_2)$ for $(\omega_1,\omega_2)\in\Omega_B\times \Omega_D$. Then $B$ and $D$, regarded as processes on the product probability space $(\Omega,\mathcal{F},\P):=(\Omega_B\times\Omega_D, \mathcal{F}_B\times \mathcal{F}_D, \P_B\times \P_D)$, are independent.
 Let $\E_B$, $\E_D$ and $\E$ denote the expectations under the probability measures $\P_B$, $\P_D$ and $\P$, respectively. Namely, for a given random element $Z$ on the product space, let 
\[
	\E_B[Z]:=\int_{\Omega_B} Z(\omega_1,\,\cdot\,)\,\ud \P_B(\omega_1), \ \ \E_D[Z]:=\int_{\Omega_D} Z(\,\cdot\,,\omega_2)\,\ud \P_D(\omega_2), 
\]
and 
\[
	\E[Z]:=\int_{\Omega} Z(\omega_1,\omega_2)\,\ud \P(\omega_1,\omega_2). 
\]
Clearly, if $Z\ge 0$ or $Z$ is $\P$-integrable, then $\E_D[\E_B[Z]]=\E[Z]$. 
\textit{We use this setting in the remainder of the paper.}

Let $E$ be the inverse of the subordinator $D$. Consider a stochastic differential equation (SDE) 
\begin{equation}
\label{SDE0}
X_t=x_0+\int_0^t H(E_r, X_r)\ud r+\int_0^t F(E_r, X_r)\ud E_r+\int_0^t G(E_r, X_r)\ud B_{E_r}, \ t\in[0,T],
\end{equation}
where $x_0\in\mathbb{R}^d$ is a non-random constant, $T>0$ is a fixed time horizon, and $H,\, F:[0,\infty)\times \mathbb{R}^d\to \mathbb{R}^d$ and $G:[0,\infty)\times \mathbb{R}^d\to \mathbb{R}^{d\times m}$ are jointly continuous functions 
such that the following assumptions hold: there exists a continuous, nondecreasing function $h(u): [0,\infty)\rightarrow [0,\infty)$ such that for all $u\ge 0$ and $x,y\in\R^d$, 
\begin{itemize}
\item[$\mathcal{H}_1$]: $|H(u,x)-H(u,y)|+|F(u,x)-F(u,y)|+|G(u,x)-G(u,y)|\leq h(u)|x-y|$;
\item[$\mathcal{H}_2$]: $|H(u, x)|+|F(u, x)|+|G(u, x)|\le h(u)(1+|x|)$, 
\end{itemize}
where $|\cdot|$ denotes the Euclidean norms of appropriate dimensions. 
\textit{In the remainder of the paper, we assume that $m=d=1$; extensions of the results established in Sections \ref{section_3} and \ref{section_4} to a multi-dimensional case are straightforward, whereas the presentation in Section \ref{section_5} is restricted to the one-dimensional case.}
Examples of coefficients satisfying the above assumptions include $H(u,x)=F(u,x)=G(u,x)=u^p x$ for some $p>0$. 
We make the bounds in Assumptions $\mathcal{H}_1$ and $\mathcal{H}_2$ time-dependent since
our convergence results for approximation schemes of SDE \eqref{SDE0} rely on the relationship between the growth of the bound $h(u)$ and that of the Laplace exponent $\psi(s)$ of the underlying subordinator $D$.
Moreover, the inclusion of the function $h(u)$ will reveal how the information about random time changes are generally retained when moments of the SDE solution are estimated. 
Note that this paper often assumes that the nondecreasing bound $h(u)$ is regularly varying at $\infty$ with positive index, in which case, the monotonicity assumption of $h(u)$ can be dropped; indeed, we can always replace $h(u)$ by the running maximum $\bar{h}(u):=\sup_{t\in[0,u]}h(t)$, which is also regularly varying at $\infty$ with the same index due to Theorem 1.5.3 of \cite{Bingham_book}. We assume the monotonicity solely for simplicity of discussions in proofs.

For each fixed $t\ge 0$, the random time $E_t$ is an $(\F_t)$-stopping time, and therefore, the time-changed filtration $(\F_{E_t})_{t\ge 0}$ is well-defined. Moreover, since the time change $E$ is an $(\F_{E_t})$-adapted process with continuous, nondecreasing paths and the time-changed Brownian motion $B\circ E$ is a continuous $(\F_{E_t})$-martingale, 
SDE \eqref{SDE0} is understood within the framework of stochastic integrals driven by continuous semimartingales (see \cite{Kobayashi} for details). 
The following lemma confirms that a unique strong solution of SDE \eqref{SDE0} exists. Its proof is analogous to the proof of Lemma 4.1 in \cite{Kobayashi}.

\begin{lemma}
SDE \eqref{SDE0} satisfying Assumption $\mathcal{H}_1$ has a unique strong solution which is a continuous $(\mathcal{F}_{E_t})$-semimartingale. 
\end{lemma}

\begin{proof}
Let $\D$ denote the space of $(\F_{E_t})$-adapted processes with c\`adl\`ag paths on $(\Omega,\F,\P)$.  Define operators $F_1,F_2,F_3:\D\to\D$ by 
$(F_1(X_\cdot))_t:=H(E_t,X_t)$, $(F_2(X_\cdot))_t$ $:=F(E_t,X_t)$, and $(F_3(X_\cdot))_t:=G(E_t,X_t)$
for $X\in\D$ and $t\ge 0$.  
Then by the joint continuity of $H$, $F$ and $G$, for each $i=1,2,3$, it follows that $X^{\sigma-}=Y^{\sigma-}$ implies $(F_i(X_\cdot))^{\sigma-}=(F_i(Y_\cdot))^{\sigma-}$ for any $X,Y\in\D$ and $(\F_{E_t})$-stopping time $\sigma$, where $Z^{\sigma-}_t:=Z_t$ if $t<\sigma$ and $Z^{\sigma-}_t:=Z_{\sigma-}$ if $t\ge \sigma$. 
Also, for all $X,Y\in\D$ and $t\ge 0$,   
$|(F_i(X_\cdot))_t-(F_i(Y_\cdot))_t|\le K_t|X_t-Y_t|$
with $K_t:=h(E_t)$.
Since $(K_t)_{t\ge 0}$ is an $(\F_{E_t})$-adapted, continuous process, each operator $F_i$ is process Lipschitz, and therefore, by Theorem 7 of Chapter V in \cite{Protter}, a unique solution of SDE \eqref{SDE0} exists and is an $(\mathcal{F}_{E_t})$-semimartingale. The continuity of the solution follows by the continuity of the driving process of the SDE. 
\end{proof}

Both $E$ and $B\circ E$ start at 0, and for quadratic variations, $[B\circ E,B\circ E]=E$ and $[E,E]=[B\circ E,E]=[B\circ E,m]=[E,m]=0$,
where $m$ denotes the identity map. 
For example, $\ud[X,X]_t=G^2(E_t,X_t)\,\ud E_t$ for the solution $X$ of SDE \eqref{SDE0}.
For details about stochastic calculus for more general time-changed semimartingales, see Section 4 in \cite{Kobayashi}.

The remainder of this section is devoted to the derivation of sufficient conditions for the existence of the $p$th moment of $\sup_{0\le r\le T}|X_r|$, where $X$ is the solution of SDE \eqref{SDE0}. 
The conditions are necessary to establish the main theorems of this paper in Sections \ref{section_4} and \ref{section_5}. 
Let us recall the Burkholder--Davis--Gundy (BDG) inequality, which states that for any $p>0$, there exists a constant $b_p>0$ such that 
\begin{align}\label{Burkholder}
	\E\left[ \sup_{0\le t\le S} |M_t|^p\right]\le b_p \E\left[ [M,M]_S^{p/2}\right]
\end{align}
for any stopping time $S$ and any continuous local martingale $M$ with quadratic variation $[M,M]$. The constant $b_p$ can be taken independently of $S$ and $M$; see Proposition 3.26 and Theorem 3.28 of Chapter 3 of \cite{KaratzasShreve}.

\begin{proposition}\label{Yp}
Let $X$ be the solution of SDE \eqref{SDE0} satisfying Assumptions $\H_1$ and $\H_2$. Assume further that one of the following two conditions holds:
\begin{enumerate}[(a)]
\item $h$ is constant; 
\item $h$ is continuous, nondecreasing, and regularly varying at $\infty$ with index $q\ge 0$, and the Laplace exponent $\psi$ of $D$ is regularly varying at $\infty$ with index $\beta\in(2q/(2q+1),1)$. 
\end{enumerate}
Then $\E[Y_T^{(p)}]<\infty$ for any $p\ge 1$, where $Y_t^{(p)}:=1+\sup_{0\leq r\leq t}|X_r|^p$.
\end{proposition}

\begin{proof}
Since a similar approach will be taken several times in Sections \ref{section_4} and \ref{section_5}, we will provide a detailed proof of this proposition. Recall the notations $\P_B,\,\P_D,\,\P,\,\E_B,\,\E_D$ and $\E$ introduced in the beginning of this section.

 Let $S_\ell:=\inf \{t\ge 0: Y^{(p)}_t>\ell\}$ for $\ell\in\mathbb{N}$. 
Since the solution $X$ has continuous paths, $Y^{(p)}_t<\infty$ for each $t\ge 0$, and hence, $S_\ell\uparrow \infty$ as $\ell\to \infty$. 
For $\P_D$-a.e.\ path, we 
first apply a Gronwall-type inequality to
the function $t\mapsto\E_B[Y^{(p)}_{t\wedge S_\ell}]$ for a fixed $\ell$ and then let $t=T$ and $\ell\to\infty$ in the obtained inequality to establish a bound for $\E_B[Y^{(p)}_T]$.  
Note that due to the definition of $S_\ell$, 
$\int_0^t \E_B[Y^{(p)}_{r\wedge S_\ell}]\,\ud E_r\le \ell E_t<\infty$, 
which allows us to safely apply the Gronwall-type inequality.

Assume $p\ge 2$ since the result for $1\le p<2$ follows immediately from the result for $p\ge 2$ with Jensen's inequality. 
By the It\^o formula, $X^p_s=x_0^p+I_{s}+J_s+K_s$, where
\begin{align*}
	&I_s:=\int_0^s pX_r^{p-1} H(E_r,X_r)\ud r; \ \ \ 
	J_s:=\int_0^s pX_r^{p-1} G(E_r,X_r)\ud B_{E_r};\\
	&K_s:=\int_0^s \left\{pX_r^{p-1} F(E_r,X_r)+\frac{1}{2}p(p-1)X_r^{p-2} G^2(E_r,X_r)\right\}\ud E_r.
\end{align*}
Fix $t\in[0,T]$ and $\ell\in\mathbb{N}$. 
By Assumption $\H_2$ and the inequality $(x+y+z)^p\le  c_p(x^p+y^p+z^p)$ for $x,y,z\ge 0$ with $c_p=3^{p-1}$, 
\begin{align}\label{2020051501}
\E_B\left[\sup_{0\leq s\leq t\wedge S_\ell}|I_s|\right]
&\leq\E_B\left[\int_0^{t\wedge S_\ell}ph(E_r)|X_r|^{p-1}(1+|X_r|)\ud r\right]\notag\\
&\leq pc_ph(E_T)\int_0^{t\wedge S_\ell}\E_B[Y^{(p)}_{r}]\ud r. 
\end{align}
Similarly,
\begin{align*}
\E_B\left[\sup_{0\leq s\leq t\wedge S_\ell}|K_s|\right]
\leq \left(pc_ph(E_T)+\frac{1}{2} p(p-1)c_ph^2(E_T)\right)\int_0^{t\wedge S_\ell}\E_B[Y^{(p)}_{r}]\ud E_r.
\end{align*}
Since $(J_s)_{s\ge 0}$ is a local martingale, applying the BDG inequality \eqref{Burkholder} yields

\noindent$\E_B\bigl[\sup_{0\leq s\leq t\wedge S_\ell}|J_s|\bigr]$
$\le b_1\E_B\bigl[\bigl(\int_0^{t\wedge S_\ell}p^2X_r^{2p-2}G^2(E_r, X_r)\ud E_r\bigr)^{1/2}\bigr] 
$,
and hence, 
\begin{align*}
\E_B\bigl[\sup_{0\leq s\leq t\wedge S_\ell}|J_s|\bigr]
&\leq b_1\E_B\left[pc_ph(E_T)\left(Y_{t\wedge S_\ell}^{(p)}\int_0^{t\wedge S_\ell}Y_r^{(p)}\ud E_r\right)^{1/2}\right]\\
&\leq \frac{1}{2}\E_B[Y_{t\wedge S_\ell}^{(p)}]+2b_1^2p^2c_p^2h^2(E_T)\int_0^{t\wedge S_\ell}\E_B[Y^{(p)}_{r}]\ud E_r, 
\end{align*}
where the last inequality follows from the elementary inequality $(ab)^{1/2}\le a/\lambda +\lambda b$ valid for any $a,b, \lambda>0$, with $\lambda:=2b_1pc_ph(E_T)$. 
(Note that $h(E_T)>0$ due to the discussion given right above Theorem \ref{theorem_negative}.)

Now, note that 
$\int_0^{t\wedge S_\ell} L_r\,\dd E_r\le \int_0^t L_{r\wedge S_\ell}\,\dd E_r$ for any nonnegative process $(L_t)_{t\ge 0}$.
Indeed, the inequality obviously holds if $t\le S_\ell$, while if $t>  S_\ell$, then 
$
	\int_0^t L_{r\wedge S_\ell}\,\dd E_r
	=\int_0^{S_\ell} L_r\,\dd E_r+\int_{S_\ell}^t L_{S_\ell}\,\dd E_r
	\ge \int_0^{t\wedge S_\ell} L_r\,\dd E_r. 
$
Thus, by the above estimates for $I_s$, $J_s$ and $K_s$,
\[
	\E_B[Y^{(p)}_{t\wedge S_\ell}]
	\le 2(1+|x_0|^p)+2pc_ph(E_T)\int_0^t\E_B[Y^{(p)}_{r\wedge S_\ell}]\ud r+2\xi (E_T)\int_0^t \E_B[Y^{(p)}_{r\wedge S_\ell}]\dd E_r,
\]
where 
$
	\xi(u):=pc_ph(u)+\left(p(p-1)c_p/2+2b_1^2p^2c_p^2\right)h^2(u).
$

With the Gronwall-type inequality in Chapter IX.6a, Lemma 6.3 of \cite{JacodShiryaev},
$	
	\E_B[Y^{(p)}_{t\wedge S_\ell}]
	\le 2(1+|x_0|^p) e^{2E_T\xi (E_T)+2pc_p T h(E_T)}.
$
Setting $t=T$, letting $\ell\to\infty$ while recalling $\xi(u)$ does not depend on $\ell$, and using the monotone convergence theorem yields 
\begin{align}\label{008}
	\E_B[Y^{(p)}_T]
	\le 2(1+|x_0|^p) e^{2E_T\xi (E_T)+2pc_pT h(E_T)}.
\end{align}

If $h$ is constant, then the right hand side of \eqref{008} takes the form $ce^{cE_T}$, so taking $\E_D$ on both sides gives  $\E[Y^{(p)}_{T}]\le \E[ce^{cE_T}]<\infty$ due to Proposition \ref{proposition_RV}. 
On the other hand, if $h\in \mathrm{RV}_q(\infty)$ with $q\ge 0$ is continuous and nondecreasing, then since $\xi(u)\in \mathrm{RV}_{2q}(\infty)$, the right hand side of \eqref{008} takes the form $ce^{f(E_T)}$ with $f(u)\in \mathrm{RV}_{2q+1}(\infty)$ 
due to Lemma \ref{lemma_RV}(ii)(iii).
So taking $\E_D$ on both sides gives $\E[Y^{(p)}_{T}]\le \E[ce^{f(E_T)}]$. By Theorem \ref{expmoment}, the latter is finite provided that $2q+1<1/(1-\beta)$, or equivalently,  $\beta\in(2q/(2q+1),1)$. 
\end{proof}

\begin{remark}\label{remark_localization}
(1) The key part in the above proof is to derive an estimate for $\E_B[Y^{(p)}_{t\wedge S_\ell}]$ rather than $\E[Y^{(p)}_{t\wedge S_\ell}]$. This makes estimations such as \eqref{2020051501} possible. Indeed, if we consider $\E[\sup_{0\le s\le t\wedge S_\ell}|I_s|]$ instead of $\E_B[\sup_{0\le s\le t\wedge S_\ell}|I_s|]$ in \eqref{2020051501}, then the expectation and integral signs are no longer interchangeable and $h(E_T)$ cannot be taken out of the integral sign, which makes the Gronwall-type inequality inapplicable. This approach is completely different from the one employed in \cite{JumKobayashi}. It also gives rise to the factor $h(E_T)$ in the exponent of the right hand side of \eqref{008}, which must be handled by Theorem \ref{expmoment}.

(2) The introduction of the localizing sequence $\{S_\ell\}$ in the above proof essentially allows us to consider the process $(Y^{(p)}_t)_{t\in [0,T]}$ to be bounded by a non-random constant.
This, in particular, guarantees that the argument based on the Gronwall-type inequality with the stochastic driver $\ud E_t$ be meaningful. 
On the other hand, in the proofs of Theorems \ref{EU}, \ref{Mil0}, \ref{Milstein} and \ref{HighEU} in Sections \ref{section_4} and \ref{section_5}, we will use a different localizing sequence given by $S_\ell:=\inf \{t\ge 0: \sup_{0\leq s\leq t}\{|X_s-X^\delta_s|\}>\ell\}$, where $X^\delta$ denotes the approximation process, in order to make the Gronwall-type inequality applicable again. However, to clarify the main ideas of the proofs, we suppress $S_\ell$ and simply give arguments assuming that 
the process $(\sup_{0\leq s\leq t}\{|X_s-X^\delta_s|\})_{t\in [0,T]}$ is bounded by a non-random constant. 
\end{remark}

\section{Rate of convergence for a Euler--Maruyama-type scheme when $H\not\equiv 0$}\label{section_4}

This section discusses the rate of strong convergence of a Euler--Maruyama-type approximation scheme for the solution of SDE \eqref{SDE0} with $H(u,x)=H(u)$ under two different sets of assumptions on the SDE coefficients in addition to Assumptions $\H_1$ and $\H_2$.  The different settings result in different rates of convergence.
The results we prove here answer Question \textbf{(A)} raised in Section \ref{section_introduction} and provide generalizations of Theorem 3.1 of \cite{JumKobayashi} to cases when the two drifts $H$ and $F$ simultaneously appear. 
However, as discussed in Section \ref{section_introduction}, the approach we take in this paper is completely different from that in \cite{JumKobayashi} as the duality principle is never used.

Let us first describe an approximation process for an inverse subordinator $E$ given in \cite{Magdziarz_simulation,Magdziarz_spa}.  Fix an equidistant step size $\delta\in (0,1)$ and a time horizon $T>0$.
To approximate $E$ on the interval $[0,T]$, we first simulate a sample path of the subordinator $D$, which has independent and stationary increments, by setting  $D_0=0$ and then following the rule $D_{i\delta}:=D_{(i-1)\delta}+Z_i$, $i=1,2,3,\ldots,$
with an i.i.d.\ sequence $\{Z_i\}_{i\in\mathbb{N}}$ distributed as $Z_i=^\mathrm{d} D_{\delta}$. 
We stop this procedure upon finding the integer $N$ satisfying 
	$T\in[D_{N\delta}, D_{(N+1)\delta}).$
 Note that the $\mathbb{N}\cup\{0 \}$-valued random variable $N$ indeed exists since  
 $D_t\to\infty$ as $t\to\infty$ a.s. 
To generate the random variables $\{Z_i\}$, one can use algorithms presented in Chapter  6 of \cite{ContTankov}. 
Next, let
$
E^\delta_t
	:=\bigl(\min\{n\in \mathbb{N}; D_{n\delta}>t\}-1\bigr)\delta.
$
The sample paths of $E^\delta=(E^\delta_t)_{t\ge 0}$ are nondecreasing step functions with constant jump size $\delta$ and the $i$th waiting time given by $Z_i=D_{i\delta}-D_{(i-1)\delta}$. Indeed, it is easy to see that for $n=0,1,2,\ldots,N$,
\[
	E^\delta_t=n\delta \ \ \textrm{whenever} \ \ t\in[D_{n\delta},D_{(n+1)\delta}).
\]
In particular, 
$E^\delta_T=N\delta.$
The process $E^\delta$ efficiently approximates $E$ as established in \cite{JumKobayashi,Magdziarz_spa}; namely, a.s., 
\begin{equation}\label{revision_1}
	E_t-\delta\le E^\delta_t\le E_t \ \ \textrm{for all} \ \  t\in[0,T]. 
\end{equation}

Now, for $n=0,1,2,\ldots,N$, let 
\[
	\tau_n=D_{n\delta}. 
\] 
By the independence assumption between $B$ and $D$,
we can approximate the Brownian motion $B$ over the time steps $\{0,\delta,2\delta,\ldots,N\delta\}$.  
With this in mind, define a discrete-time process $(X^\delta_{\tau_n})_{n\in \{0,1,2,\ldots,N\}}$ by $X^\delta_0:=x_0$ and for $n=0,1,2,\ldots,N-1$,
\begin{align}\label{revision_2}
X^\delta_{\tau_{n+1}}:&=X^\delta_{\tau_{n}}
+H(E^\delta_{\tau_{n}})(\tau_{n+1}-\tau_n)+F(E^\delta_{\tau_{n}},X_{\tau_n}^\delta)(E^\delta_{\tau_{n+1}}-E^\delta_{\tau_{n}})\notag\\
&\ \ \ +G(E^\delta_{\tau_{n}}, X^\delta_{\tau_n})(B_{E^\delta_{\tau_{n+1}}}-B_{E^\delta_{\tau_{n}}}),
\end{align}
which, due to the relationship $E^\delta_{\tau_{n}}=n\delta$, is equivalent to 
\begin{align}\label{def-EM}
X^\delta_{\tau_{n+1}}\!\!:=\!X^\delta_{\tau_{n}}
\!+\!H(n\delta)(\tau_{n+1}-\tau_n)\!+\!F(n\delta,X_{\tau_n}^\delta)\delta\!+\!G(n\delta, X^\delta_{\tau_n})(B_{(n+1)\delta}-B_{n\delta}).
\end{align}
In particular, if $H\equiv F\equiv 0$ and $G\equiv 1$, then $(X^\delta_{\tau_n})_{n\in \{0,1,2,\ldots,N\}}$ becomes a discretized time-changed Brownian motion whose sample path is generated as in Figure \ref{figure_001}. 
Note that even though expression \eqref{def-EM} might look as though non-random time steps were taken, we indeed take the random time steps $\tau_0, \tau_1, \tau_2, \ldots, \tau_N$ to discretize the driving processes $E=(E_t)_{t\ge 0}$ and $B\circ E=(B_{E_t})_{t\ge 0}$ as indicated in \eqref{revision_2}; thus, the key characteristic of the random trapping events of the time-changed process (which give rise to constant periods) is in fact captured by the random step sizes $\tau_{n+1}-\tau_n=D_{(n+1)\delta}-D_{n\delta}=^\mathrm{d}D_\delta$. 
Note also that we do not discretize the SDE via non-random time steps; that would be practically difficult since the driving processes $E$ and $B\circ E$ have neither independent nor stationary increments.

To define a continuous-time process $(X^\delta_t)_{t\in[0,T]}$, we adopt the continuous interpolation; i.e.\ whenever $s\in[\tau_n,\tau_{n+1})$, 
\begin{align}\label{011}
X^\delta_s&:=X^\delta_{\tau_{n}}+\int_{\tau_n}^{s}H(E_{\tau_n})\,\ud r+\int_{\tau_n}^{s}F(E_{\tau_n},X_{\tau_n}^\delta)\,\ud E_r+\int_{\tau_n}^{s}G(E_{\tau_n}, X^\delta_{\tau_n})\,\ud B_{E_r}.
\end{align}
Let
\[
	n_t=\max\{n\in \mathbb{N}\cup \{0\}; \tau_n\le t\} \ \ \textrm{for} \ \ t\ge 0.
\]
Then clearly $\tau_{n_t}\le t<\tau_{n_t+1}$ for any $t>0$. 
Using \eqref{def-EM} and the identity
$X^\delta_s-x_0=\sum_{i=0}^{n_s-1}(X^\delta_{\tau_{i+1}}-X^\delta_{\tau_i})+(X^\delta_s-X^\delta_{\tau_{n_s}})$, we can express $X^\delta_s-x_0$ as 
\begin{align*}
	\sum_{i=0}^{n_s-1}\!\!\left[H(E_{\tau_i})(\tau_{i+1}\!-\!\tau_i) \!+\!F(E_{\tau_i},X^\delta_{\tau_i})\delta \!+\! G(E_{\tau_i}, X^\delta_{\tau_i})(B_{(i+1)\delta}\!-\!B_{i\delta})\right]\!+\!(X^\delta_s\!-\!X^\delta_{\tau_{n_s}}\!), 
\end{align*}
where we used $i\delta=E_{D_{i\delta}}=E_{\tau_i}$. 
Using \eqref{011} and the fact that $\tau_i=\tau_{n_r}$ for any $r\in[\tau_i,\tau_{i+1})$, we can rewrite the latter in the convenient form
\begin{align}\label{010}
X^\delta_s&=x_0+\int_{0}^{s}H(E_{\tau_{n_r}})\,\ud r+\int_{0}^{s}F(E_{\tau_{n_r}},X_{\tau_{n_r}}^\delta)\,\ud E_r+\int_{0}^{s}G(E_{\tau_{n_r}}, X^\delta_{\tau_{n_r}})\,\ud B_{E_r}.
\end{align}

We are now ready to state the first main theorem of this paper, where we assume that there exist constants $K>0$ and $\theta\in(0,1]$ such that for all $u,v\ge 0$ and $x\in \R$,  
\begin{itemize}
\item[$\H_3$]: $|H(u)-H(v)|+|F(u,x)-F(v,x)|+|G(u,x)-G(v,x)|\leq K|u-v|^\theta(1+|x|)$.
\end{itemize}
Recall that an approximation process $X^\delta$ with step size $\delta>0$ is said to \textit{converge strongly to the solution $X$ uniformly on $[0,T]$ with order $\eta\in(0,\infty)$} if there exist finite positive constants $C$ and $\delta_0$ such that for all $\delta\in(0,\delta_0)$, $\mathbb{E}\left[\sup_{0\le t\le T}|X_t-X^\delta_t|\right]\le C\delta^\eta$.

\begin{theorem}
\label{EU} 
Let $X$ be the solution of SDE \eqref{SDE0} with $H(u,x)=H(u)$ for all $(u,x)\in[0,T]\times \R$ such that Assumptions $\H_1$, $\H_2$ and $\H_3$ hold. 
Assume that the Laplace exponent $\psi$ of $D$ is regularly varying at $\infty$ with index $\beta\in(0,1)$ and that one of the following conditions holds:
\begin{enumerate}[(a)]
\item $h$ is constant and $\beta\in(1/2,1)$;
\item $h$ is continuous, nondecreasing, and regularly varying at $\infty$ with index $q\ge 0$, and $\beta\in \left((2q+1)/(2q+2),1\right).$
\end{enumerate}
Let $X^\delta$ be the approximation process of Euler--Maruyama-type defined in \eqref{def-EM}-\eqref{011}. 
Then there exists a constant $C>0$ not depending on $\delta$ such that for all $\delta\in(0,1)$,
$
	\E\left[\sup_{0\leq s\leq T}|X_s-X^\delta_s|\right]\leq C\delta^{\min\{\theta,1/2\}}.
$
Thus, $X^\delta$ converges strongly to $X$ uniformly on $[0,T]$ with order $\min\{\theta,1/2\}$.
\end{theorem}

The proof of this theorem relies on the following lemma, which can be viewed as a generalization of Lemma 3.2 in \cite{JumKobayashi} to cases when $H\not\equiv 0$. Recall the notations $\P_B,\,\P_D,\,\P,\,\E_B,\,\E_D$ and $\E$ defined in the beginning of Section \ref{section_3}.

\begin{lemma}
\label{XX}
Under the assumptions of Theorem \ref{EU}, for any $t\ge s\ge0$,
\[
\E_B[|X_t-X_s|]\leq \sqrt 2h(E_t)\E_B[Y_t^{(2)}]^{1/2}\big\{(t-s)+(E_t-E_s)^{1/2}+(E_t-E_s)\big\},
\]
where $Y_t^{(2)}$ is defined in Proposition \ref{Yp}. 
\end{lemma}

\begin{proof}
By the Cauchy-Schwartz inequality, $\E_B[|X_t-X_s|]$ is dominated above by 
\begin{align*}
&\E_B\!\left[\int_s^t\! |H(E_r)|\ud r\right]\!+\!\E_B\left[\int_s^t \! |F(E_r, X_r)|\ud E_r\right]\!+\!\E_B\left[\left|\int_s^t\!G(E_r, X_r)\ud B_{E_r}\right|^2\right]^{1/2}\\
&\leq (t-s)h(E_t)+(E_t-E_s)h(E_t)\E_B[Y_t^{(1)}]+\sqrt 2(E_t-E_s)^{1/2} h(E_t)\E_B[Y_t^{(2)}]^{1/2}.
\end{align*}
Jensen's inequality gives the desired bound.
\end{proof}

Now we are ready to prove Theorem \ref{EU}, where the localizing sequence $S_\ell:=\inf \{t\ge 0: \sup_{0\leq s\leq t}\{|X_s-X^\delta_s|\}>\ell\}$ allows us to consider the process $(\sup_{0\leq s\leq t}\{|X_s-X^\delta_s|\})_{t\in [0,T]}$ to be bounded. However, as stated in Remark \ref{remark_localization}(2), in order to clarify the main ideas of the proof, we suppress $S_\ell$ and simply give arguments assuming the boundedness. The same argument applies to the proofs of all the statements in Sections \ref{section_4} and \ref{section_5}. \vspace{3mm}

\noindent\textit{Proof of Theorem \ref{EU}}\ \ 
Let $Z_t:=\sup_{0\leq s\leq t}|X_s-X^\delta_s|$ for $t\in[0,T]$. Then by the representations  of $X$ and $X^\delta$ in \eqref{SDE0} and \eqref{010}, respectively, $Z_t\leq I_1+I_2+I_3$, where
\begin{align*}
I_1&:=
\sup_{0\leq s\leq t}\left|\int_0^s (H(E_r)-H(E_{\tau_{n_r}}))\ud r\right|;\\
I_2&:=\sup_{0\leq s\leq t}\left|\int_0^s (F(E_r, X_r)-F(E_{\tau_{n_r}}, X_{\tau_{n_r}}^\delta))\ud E_r\right|;\\
I_3&:=\sup_{0\leq s\leq t}\left|\int_0^s (G(E_r, X_r)-G(E_{\tau_{n_{r}}}, X_{\tau_{n_{r}}}^\delta))\ud B_{E_r}\right|.
\end{align*}

In terms of $I_1$, recall that $\tau_{n_r}\le r<\tau_{n_r+1}$ and $0\le E_r-E_{\tau_{n_r}}\le (n_r+1)\delta-n_r \delta=\delta$. This, together with the Cauchy-Schwartz inequality and Assumption $\H_3$, yields
\begin{align}\label{012}
I_1^2&\leq t\int_0^t(H(E_r)-H(E_{\tau_{n_r}}))^2\ud r\leq K^2T^2\delta^{2\theta}.
\end{align}

As for $I_2$, note that
$\E_B[I_2^2]\leq E_t\int_0^t\E_B\bigl[(F(E_r, X_r)-F(E_{\tau_{n_r}}, X_{\tau_{n_r}}^\delta))^2\bigr]\ud E_r$ by the Cauchy-Schwartz inequality. Assumptions $\H_1$ and  $\H_3$ together with the inequality $|F(E_r, X_r)-F(E_{\tau_{n_r}}, X_{\tau_{n_r}}^\delta)|\le |F(E_r, X_r)-F(E_{\tau_{n_r}}, X_r)|+|F(E_{\tau_{n_r}}, X_r)-F(E_{\tau_{n_r}}, X_{\tau_{n_r}})|+|F(E_{\tau_{n_r}}, X_{\tau_{n_r}})-F(E_{\tau_{n_r}}, X_{\tau_{n_r}}^\delta)|$ yield
\begin{align}
\label{I2}
&\E_B[I_2^2]\notag\\
&\leq 3E_T\int_0^t \E_B\bigl[K^2\delta^{2\theta}(1+|X_r|)^2+h^2(E_{\tau_{n_r}})|X_r-X_{\tau_{n_r}}|^2+h^2(E_{\tau_{n_r}}) Z_{\tau_{n_r}}^2
\bigr]\ud E_r\notag\\
&\leq 6E_T^2K^2\delta^{2\theta}\E_B[Y_T^{(2)}]\notag\\
&\ \ +3E_Th^2(E_T)
\biggl\{\int_0^t\E_B[|X_r-X_{\tau_{n_r}}|^2]\ud E_r+
\int_0^t\E_B[Z_r^2]\ud E_r\biggr\}.
\end{align}
Now, for any $r\in[0,t]$, by Lemma \ref{XX} and the fact that $0\le E_r-E_{\tau_{n_r}}\leq \delta$, 
\begin{align}
\label{61}
\E_B[|X_r-X_{\tau_{n_r}}|^2]\leq 6h^2(E_T)\E_B[Y_T^{(2)}]\big\{(r-\tau_{n_r})^2+\delta+\delta^2\big\}.
\end{align}
Moreover,
\begin{align}
\label{62}
&\int_0^t(r-\tau_{n_r})^2\ud E_r=\sum_{i=0}^{n_t-1}\int_{\tau_i}^{\tau_{i+1}}(r-\tau_i)^2\ud E_r+\int_{\tau_{n_t}}^t(r-\tau_{n_t})^2\ud E_r\notag\\
&\leq \delta\left(\sum_{i=0}^{n_t-1}(\tau_{i+1}-\tau_i)^2+(t-\tau_{n_t})^2\right)
\le 2T\delta \left(\sum_{i=0}^{n_t-1}(\tau_{i+1}-\tau_i)+(t-\tau_{n_t})\right)\notag\\
&\leq 2T^2\delta.
\end{align}
Combining \eqref{61} and \eqref{62} with \eqref{I2} and recalling that $\delta<1$ gives 
\begin{align}\label{013}
\E_B[I_2^2]\leq \xi_1(E_T)\E_B[Y_T^{(2)}]\delta^{\min\{2\theta,1\}}+3E_Th^2(E_T)\int_0^t\E_B[Z_r^2]\ud E_r,
\end{align}
where $\xi_1(u):=36u^2h^4(u)+36uh^4(u)T^2+6u^2K^2$. 

In terms of $I_3$, the BDG inequality \eqref{Burkholder} and a calculation similar to the previous paragraph yield
\begin{align}\label{014}
\E_B[I_3^2]
&\leq \xi_2(E_T)\E_B[Y_T^{(2)}]\delta^{\min\{2\theta,1\}}+3b_2h^2(E_T)\int_0^t\E_B[Z_r^2]\ud E_r,
\end{align}
where $\xi_2(u):=b_2\xi_1(u)/u$.

Putting together estimates \eqref{012}, \eqref{013} and \eqref{014} gives
$$
\E_B[Z_t^2]\leq \xi_3(E_T)\E_B[Y_T^{(2)}]\delta^{\min\{2\theta,1\}}+ 9(E_T+b_2)h^2(E_T)\int_0^t\E_B[Z_r^2]\ud E_r, 
$$
where $\xi_3(u):=3\xi_1(u)+3\xi_2(u)+3K^2T^2$.
Using the Gronwall-type inequality in Chapter IX.6a, Lemma 6.3 of \cite{JacodShiryaev} and setting $t=T$ gives
$$
\E_B[Z_T^2]\leq \xi_3(E_T)\E_B[Y_T^{(2)}] e^{9E_T(E_T+b_2)h^2(E_T)}\delta^{\min\{2\theta,1\}}.
$$
Taking $\E_D$ on both sides and using the Cauchy-Schwartz inequality, 
\begin{align}\label{022}
\E[Z_T^2]\leq \E[\xi_3^4(E_T)]^{1/4}\E[(Y_T^{(2)})^4]^{1/4} \E\left[e^{18E_T(E_T+b_2)h^2(E_T)}\right]^{1/2}\delta^{\min\{2\theta,1\}}.
\end{align}
The desired result follows upon showing the expectations on the right hand side of \eqref{022} are finite.  
Now, suppose $h$ is constant. Then it follows from Propositions \ref{proposition_RV} and \ref{Yp} that $\E[\xi_3^4(E_T)]<\infty$ and $\E[(Y_T^{(2)})^4]\leq 8\E[Y_T^{(8)}]<\infty$. In terms of $\E\bigl[e^{18E_T(E_T+b_2)h^2(E_T)}\bigr]$, the exponent takes the form $f(E_T)$ with $f\in \mathrm{RV}_{2}(\infty)$ due to Lemma \ref{lemma_RV}, so the expectation is finite if $2<1/(1-\beta)$ (or equivalently, $\beta\in(1/2,1)$) due to Theorem \ref{expmoment}.

On the other hand, if $h\in \mathrm{RV}_q(\infty)$ with $q\ge 0$, then by Propositions  \ref{Yp}, $\E[Y_T^{(8)}]<\infty$ if $\beta\in(2q/(2q+1),1)$. Since the exponent of $\E\bigl[e^{18E_T(E_T+b_2)h^2(E_T)}\bigr]$ takes the form $f(E_T)$ with $f\in \mathrm{RV}_{2q+2}(\infty)$, the expectation is finite if $2q+2<1/(1-\beta)$ (or equivalently, $\beta\in((2q+1)/(2q+2),1)$) due to Theorem \ref{expmoment}. Consequently, the result follows as long as $\beta\in((2q+1)/(2q+2),1)$. \qed 

\begin{remark}
(1) The method used in the above proof provides a general idea of how to analyze the rates of strong convergence of approximation schemes for possibly larger classes of SDEs involving random time changes.

(2) The above proof would not work if we allowed the coefficient $H$ to depend on $x$ as well. Indeed, that case would require an estimation of $\E_B[I_1^2]$ in a way similar to \eqref{I2}, giving rise to the integral $\int_0^t \E_B[|X_r-X_{\tau_{n_r}}|^2]\,\ud r$. This integral, due to \eqref{61}, could be dominated above by a quantity involving the integral  $\int_0^t(r-\tau_{n_r})^2\,\ud r$. However, unlike \eqref{62}, the latter integral 
would not yield a bound containing $\delta$. 

(3) Although we interpolated the discretized process $(X^\delta_{\tau_n})_{n\in\{0,1,2,\ldots,N\}}$ via the continuous interpolation in \eqref{011}, it is also possible to adopt the piecewise constant interpolation $X^\delta_t:=X^\delta_{\tau_{n_t}}$ as in \cite{JinKobayashi} when $H\equiv 0$.
In the latter case, the bound for $Z_t$ will additionally contain the suprema of integrals over $[\tau_{n_s},s]$, including $I_5:=\sup_{0\le s\le t}|\int_{\tau_{n_s}}^s G(E_r,X_r)\,\ud B_{E_r}|$. 
Estimation of $\E_B[I_5^2]$ can be carried out with the help of
a result on modulus of continuity of stochastic integrals established in \cite{FischerNappo}, which only yields the convergence order $1/2-\ve$ for any $\ve>0$. 
See Remark 9(3) of \cite{JinKobayashi} for details.

(4) The above proof shows that the parameter $\beta$ plays an important role in determining the finiteness of the upper bound for $\E[Z_T^2]$. On the other hand, $\beta$ does not affect the rate of convergence for $X^\delta$, which is due to the fact that we constructed the discretized time change $E^\delta$ in such a way that $E^\delta$ converges to $E$ with order 1 \textit{regardless of the value of $\beta$,} as indicated by \eqref{revision_1}. The above argument shows that the rate of convergence for $X^\delta$ could possibly involve $\beta$ if the rate of convergence for $E^\delta$ depended on $\beta$.
\end{remark}

We now consider SDE \eqref{SDE0} when not only $H(u,x)$ but also $G(u,x)$ is independent of $x$.
The order of convergence cannot be increased if we use the Euler--Maruyama-type approximation since the term $\E_B[I_2^2]$ in the proof of Theorem \ref{EU} remains the same and gives a bound containing $\delta^{\min\{2\theta,1\}}$.
In order to estimate $\E_B[I_2^2]$ in a way that a bound involves the better rate $\delta^{\min\{2\theta,2\}}$ ($=\delta^{2\theta}$),
we employ a Milstein-type approach, which assumes some differentiability on $F$ and uses the It\^o formula to expand it. 
In addition to Assumptions $\H_1$ and $\H_2$, we assume that there exist constants $K>0$ and $\theta\in(0,1]$ and a continuous, nondecreasing function $k(u):[0,\infty)\rightarrow(0,\infty)$ such that for all $u,v\ge 0$ and $x\in \R$, 
\begin{itemize}
\item[$\H_4$:] $\bullet$ $F\in C^{1,2}$;
\item[\phantom{$\H_4$:}] $\bullet$ $|H(u)-H(v)|+|G(u)-G(v)|\leq K|u-v|^\theta$;
\item[\phantom{$\H_4$:}] $\bullet$ $|F_u(u,x)|+|F_xH(u,x)|+|F_xF(u,x)|+|F_xG(u,x)|+|F_{xx}G^2(u,x)|$
\item[\phantom{$\H_4$:}] $\ \ \ \leq k(u)(1+|x|)$.
\end{itemize}
Here, we introduce a new function $k(u)$ in addition to the already given function $h(u)$ since $k(u)$ and $h(u)$ will differently affect the range of $\beta$ values for which our argument works. Note that even though our approach is of Milstein-type, since $G_x(u,x)\equiv 0$ when $G(u,x)$ does not depend on $x$, the approximation scheme itself is no different from the Euler--Maruyama-type scheme. 
In this restrictive setting with $\theta>1/2$, the order of strong convergence improves as the following theorem shows.

\begin{theorem}
\label{Mil0} 
Let $X$ be the solution of SDE \eqref{SDE0} with $H(u,x)=H(u)$ and $G(u,x)=G(u)$ for all $(u,x)\in[0,T]\times \R$ such that Assumptions $\H_1$, $\H_2$ and $\H_4$ hold. Assume further that one of the following conditions holds: 
\begin{enumerate}[(a)]
\item $h$ and $k$ are constant; 
\item $h$ and $k$ are continuous and nondecreasing, $h$ and $\log k$ are regularly varying at $\infty$ with indices $q\ge 0$ and $\tilde{q}\ge 0$, respectively, and the Laplace exponent $\psi$ of $D$ is regularly varying at $\infty$ with index 
$\beta\in \left((q_\ast-1)/q_\ast,1\right)$, where $q_\ast:=\max\{2q+1,\tilde{q}\}$. 
\end{enumerate}
Let $X^\delta$ be the approximation process of Euler--Maruyama-type defined in \eqref{def-EM}-\eqref{011}. 
Then there exists a constant $C>0$ not depending on $\delta$ such that for all $\delta\in(0,1)$,
$
	\E\left[\sup_{0\leq s\leq T}|X_s-X^\delta_s|\right]\leq C\delta^\theta.
$
Thus, $X^\delta$ converges strongly to $X$ uniformly on $[0,T]$ with order $\theta$.
 \end{theorem}

\begin{proof}
Using \eqref{SDE0} and expanding the integrand of the $\ud E_r$ integral 
via the It\^o formula,
\begin{align*}
X_{\tau_{n+1}}
&=X_{\tau_{n}}+\int_{\tau_n}^{\tau_{n+1}}H(E_r)\ud r+\int_{\tau_n}^{\tau_{n+1}} F(E_{\tau_n}, X_{\tau_n})\ud E_r+\int_{\tau_n}^{\tau_{n+1}} G(E_r)\ud B_{E_r}\\
&\ \ +R_{(\tau_n,\tau_{n+1})};\\
R_{(a,b)}
&:=\int_{a}^{b}\int_{a}^{r_2}F_xH\ud r_1\ud E_{r_2}+\int_{a}^{b}\int_{a}^{r_2}\left(F_u+F_xF+\frac{1}{2}F_{xx}G^2\right)\ud E_{r_1}\ud E_{r_2}\\
&+\int_{a}^{b}\int_{a}^{r_2}F_xG\ud B_{E_{r_1}}\ud E_{r_2},
\end{align*}
with all the integrands evaluated at $(E_{r_1}, X_{r_1})$.
This representation together with representation \eqref{010} for $X^\delta$ gives 
$Z_t:=\sup_{0\leq s\leq t}|X_s-X^\delta_s|\leq I_1+I_2+I_3+I_4$, where
\begin{align*}
&I_1:=
\sup_{0\leq s\leq t}\left|\int_0^s \{H(E_r)-H(E_{\tau_{n_r}})\}\ud r\right|;\\
&I_2:=
\sup_{0\leq s\leq t}\left|\int_0^s \{F(E_{\tau_{n_r}}, X_{\tau_{n_r}})-F(E_{\tau_{n_r}}, X_{\tau_{n_r}}^\delta)\}\ud E_r\right|;\\
&I_3:=\!\!\sup_{0\leq s\leq t}\left|\int_0^s \{G(E_r)-G(E_{\tau_{n_r}})\}\ud B_{E_r}\right|;
\ I_4\!:=\!\!\sup_{0\leq s\leq t}\left|\sum_{i=0}^{n_s-1}\!\!R_{(\tau_i,\tau_{i+1})}+R_{(\tau_{n_s},s)}\right|.
\end{align*}
It is easy to observe that 
\begin{align}\label{017}
	I_1\leq KT\delta^{\theta};\ \ \E_B[I_2]\leq h(E_T)\int_0^t \E_B[Z_r]\ud E_r; \ \ 
	\E_B[I_3]\leq\E_B[I_3^2]^{1/2}\leq 2KE_T\delta^{\theta}.
\end{align}
The main technical part concerns the remainder term $I_4$, which contains double integrals involving three different integrators: $\ud r_1\ud E_{r_2}$, $\ud E_{r_1}\ud E_{r_2}$ and $\ud B_{E_{r_1}}\ud E_{r_2}$. We will deal with them one by one below.

In terms of the term driven by $\ud r_1\ud E_{r_2}$, by Assumption $\H_4$ and a discussion similar to \eqref{62}, 
\begin{align}
\E_B\!\!\left[\sup_{0\leq s\leq t}\left|\int_0^s\!\!\int_{\tau_{n_{r_2}}}^{r_2}\!\!\!\!F_xH(E_{r_1}, X_{r_1})\ud r_1\ud E_{r_2}\right|\right]
&\leq 
k(E_T)\E_B[Y_T^{(1)}]\int_0^t\!\!(r_2-\tau_{n_{r_2}})\ud E_{r_2}\notag\\
&\leq Tk(E_T)\E_B[Y_T^{(1)}]\delta,\label{t1}
\end{align}
where $Y_T^{(1)}$ is defined in Proposition \ref{Yp}.
Similarly, as for the second integrator $\ud E_{r_1}\ud E_{r_2}$, 
\begin{align}
\label{t2}
&\E_B\!\left[\sup_{0\leq s\leq t}\left|\int_0^s\!\!\int_{\tau_{n_{r_2}}}^{r_2}\!\!\!\!\!\big\{\!F_u(E_{r_1}\!,\! X_{r_1})\!+\!F_xF(E_{r_1}\!,\! X_{r_1})\!+\!\frac{1}{2}F_{xx}G^2(E_{r_1}\!,\!X_{r_1})\!\big\}\ud E_{r_1}\ud E_{r_2}\right|\right]\notag\\
&\leq \frac{5}{2}k(E_T)\E_B[Y_T^{(1)}]\int_0^t\int_{\tau_{n_{r_2}}}^{r_2}\ud E_{r_1}\ud E_{r_2}\leq \frac{5}{2}E_Tk(E_T)\E_B[Y_T^{(1)}]\delta.
\end{align}
On the other hand, the term driven by $\ud B_{E_{r_1}}\ud E_{r_2}$ requires a careful discussion. We need to estimate $\E_B[\sup_{0\le s\le t}|M_{n_s}+U_s|]$,
where $M_0:=0$, $M_n:=\sum_{i=0}^{n-1}L_i$ for $n\ge 1$,
\[
L_i\!\!:=\!\!\int_{\tau_i}^{\tau_{i+1}}\!\!\int_{\tau_i}^{r_2}\!\!F_xG(E_{r_1}, X_{r_1})\ud B_{E_{r_1}}\!\ud E_{r_2};
\ U_s\!\!:=\!\!\int_{\tau_{n_s}}^{s}\!\!\int_{\tau_{n_s}}^{r_2}\!\!F_xG(E_{r_1}, X_{r_1})\ud B_{E_{r_1}}\!\ud E_{r_2}.
\] 
We first verify that the stochastic integrals $L_i$, $i=0,1,\ldots,n_t-1$, are uncorrelated with respect to $\P_B$. Let $i<j$, so that $\tau_{i+1}\le \tau_j$. Observe that $\E_B[L_iL_j]=\E_B[L_i\E_B[L_j|\mathcal{F}_{E_{\tau_j}}]]$.
By assumption and estimate \eqref{008}, 
\[
\E_B\left[\int_{\tau_j}^{\tau_{j+1}}\left|\int_{\tau_j}^{r_2}F_xG(E_{r_1}, X_{r_1})\ud B_{E_{r_1}}\right|^2\ud E_{r_2}\right]
\leq \delta^2k^2(E_t)\E_B[Y^{(2)}_t]<\infty.
\]
Thus, 
$
\E_B[L_j|\mathcal{F}_{E_{\tau_j}}]
=\int_{\tau_j}^{\tau_{j+1}}\E_B\bigl[\int_{\tau_j}^{r_2}F_xG(E_{r_1}, X_{r_1})\ud B_{E_{r_1}}\big|\mathcal{F}_{E_{\tau_j}}\bigr]\ud E_{r_2}=0
$ due to the conditional Fubini theorem (Theorem 27.17 in \cite{Schilling}) and the martingale property,
thereby yielding the uncorrelatedness.
On the other hand, since $E$ has continuous paths, 
the change-of-variables formula (Theorem 3.1 in \cite{Kobayashi}) implies that $M_n$ can be expressed as
$
\sum_{i=0}^{n-1}\int_{i\delta}^{(i+1)\delta}\int_{i\delta}^{E_{r_2}}F_xG(r_1, X_{D_{r_1-}})\ud B_{r_1}\ud r_2. 
$
The latter representation, together with the proofs of Lemmas 5.7.1 and 10.8.1 of \cite{KloedenPlaten}, shows that the discrete-time process $(M_n)_{n\ge 0}$  
is a square-integrable, $((\mathcal{F}_{n\delta})_{n\ge 0},\P_B)$-martingale starting at 0. 
Therefore, by the BDG inequality \eqref{Burkholder} and the uncorrelatedness of $L_i$'s, 
$
\E_B[\sup_{0\leq s\leq t}M_{n_s}^2]
	\le b_2\sum_{i=0}^{n_t-1}\E_B[L_i^2]; 
$
hence, by the Cauchy-Schwartz inequality, 
\begin{align}
	&\E_B\left[\sup_{0\leq s\leq t}M_{n_s}^2\right]\leq b_2\delta\sum_{i=0}^{n_t-1}\int_{\tau_i}^{\tau_{i+1}}\E_B\left[\int_{\tau_i}^{r_2}\left|F_xG(E_{r_1}, X_{r_1})\right|^2\ud E_{r_1}\right]\ud E_{r_2}\notag\\
	&\leq 2b_2\delta k^2(E_T)\E_B[Y_T^{(2)}]\sum_{i=0}^{n_t-1}\int_{\tau_i}^{\tau_{i+1}}(E_{r_2}-E_{\tau_i})\ud E_{r_2}
	\leq 2b_2E_Tk^2(E_T)\E_B[Y_T^{(2)}]\delta^2.\label{0002}
\end{align}
On the other hand, 
\begin{align}
	\E_B\left[\sup_{0\le s\le t}U_s^2\right]
	&\!\le \E_B\bigg[\!\sup_{0\leq s\leq t}(E_s-E_{\tau_{n_s}})\int_{\tau_{n_s}}^{s}\left|\int_{\tau_{n_s}}^{r_2}F_xG(E_{r_1}, X_{r_1})\ud B_{E_{r_1}}\right|^2\!\!\ud E_{r_2}\bigg]\notag\\
	&\!\leq \delta\int_0^t\E_B\left[\sup_{s\in[r_2, t]}\left|\int_{\tau_{n_s}}^{r_2}F_xG(E_{r_1}, X_{r_1})\ud B_{E_{r_1}}\right|^2\right]\ud E_{r_2}. \label{0001}
\end{align}
Since $\{(\tau_{n_s},r_2): r_2\leq s\leq t\}\subset\{(\tau_{n_{r_2}},u): \tau_{n_{r_2}}\leq u\leq r_2\}$, the integrand 

\noindent$\E_B\bigl[\sup_{s\in[r_2,t]}\bigl|\int_{\tau_{n_s}}^{r_2}F_xG(E_{r_1}, X_{r_1})\ud B_{E_{r_1}}\bigl|^2\bigl]$ is less than or equal to
\begin{align*}
&\E_B\bigg[\sup_{u\in[\tau_{n_{r_2}},r_2]}\bigg|\int_{\tau_{n_{r_2}}}^{u}F_xG(E_{r_1}, X_{r_1})\ud B_{E_{r_1}}\bigg|^2\bigg]\\
&\leq b_2\E_B\left[\int_{\tau_{n_{r_2}}}^{r_2}|F_xG(E_{r_1}, X_{r_1})|^2\ud E_{r_1}\right]\leq 2b_2k^2(E_T)\E_B[Y_T^{(2)}]\delta.
\end{align*}
Hence, \eqref{0001} is dominated above by $2b_2E_Tk^2(E_T)\E_B[Y_T^{(2)}]\delta^2.$
Putting this together with \eqref{0002} yields
\begin{align}\label{0003}
\E_B\left[\sup_{0\leq s\leq t}|M_{n_s}+U_s|^2\right]
\le 8b_2E_Tk^2(E_T)\E_B[Y_T^{(2)}]\delta^2.
\end{align}
By estimates \eqref{t1}, \eqref{t2} and \eqref{0003}, 
\begin{align}\label{018}
	\E_B[I_4]\leq \bigg\{T\E_B[Y_T^{(1)}]+\frac{5}{2}E_T\E_B[Y_T^{(1)}]+(8b_2E_T\E_B[Y_T^{(2)}])^{1/2}\bigg\}
k(E_T)\delta.
\end{align}

Now, combining \eqref{017} and \eqref{018} with $\E_B[Y_T^{(1)}]\leq \sqrt 2\E_B[Y_T^{(2)}]^{1/2}$ yields
$
	\E_B[Z_t]\leq \xi_4(E_T)\E_B[Y_T^{(2)}]^{1/2}\delta^\theta+h(E_T)\int_0^t \E_B[Z_r]\ud E_r,
$
 where
$\xi_4(u):=KT+2Ku+(\sqrt 2T+(5\sqrt 2/2)u+(8b_2u)^{1/2})
k(u)$. 
Applying the Gronwall-type inequality, taking $\E_D$ on both sides, and using the Cauchy-Schwartz inequality gives
\[
	\E[Z_T]\leq \E[\xi_4^4(E_T)]^{1/4}\E[(Y_T^{(2)})^2]^{1/4}\E[e^{2E_Th(E_T)}]^{1/2}\delta^\theta. 
\]
If $h$ and $k$ are constant, then the latter bound is finite due to Propositions \ref{proposition_RV} and \ref{Yp}, thereby yielding the desired result.  
Suppose now that $h$ and $\log k$ are regularly varying and 
note that the function $k^4(u)$ involved in $\xi_4^4(u)$ can be expressed as $e^{4\log k(u)}$. 
Then Theorem \ref{expmoment} and Proposition \ref{Yp} together guarantee the finiteness of the latter bound provided that $\beta>q/(q+1)$, $\beta>(\tilde{q}-1)/\tilde{q}$ and $\beta>2q/(2q+1)$, which requires the value of $\beta$ be restricted as stated in the theorem.   
\end{proof}

\begin{remark}
In the above proof when $h$ and $\log k$ are regularly varying, analyzing the first moment $\E_B[Z_r]$ instead of the second moment $\E_B[Z_r^2]$
is crucial since estimation of $\E_B[Z_r^2]$ would give us a similar convergence result with $q_\ast$ replaced by the larger value $\bar{q}_\ast=\max\{2q+2,\tilde{q}\}$, yielding the result with a narrower range of $\beta$ values. 
Indeed, if we attempted to deal with the second moment $\E_B[Z_r^2]$, we would rely on the estimate $\E_B[I_2^2]\leq E_Th^2(E_T)\int_0^t \E_B[Z_r^2]\ud E_r$ instead of the estimate for $\E_B[I_2]$ appearing in \eqref{017}. This would eventually lead to a bound for $\E[Z_r^2]$ containing a quantity of the form $\E[e^{cE_T^2h^2(E_T)}]$, which is finite for $\beta> (2q+1)/(2q+2)$ only.
\end{remark}

\section{Rates of convergence for It\^o--Taylor-type schemes when $H\equiv 0$}\label{section_5}

This section answers Question \textbf{(B)} raised in Section \ref{section_introduction}. The presentation is restricted to the one-dimensional case since extensions to a multi-dimenional case are not straightforward; the non-commutativity of noises makes a significant difference in the analysis. 
We consider the time-changed SDE \eqref{SDE0} in which a drift with the non-random integrator is not present (i.e.\ $H\equiv 0$)
but in which the coefficient $G(u,x)$ depends on $x$, unlike Theorem \ref{Mil0}. 
Theorem 3.1 of \cite{JumKobayashi} established that the rate of strong convergence for the Euler--Maruyama-type scheme for such an SDE is $1/2-\ve$ for any small $\ve$. A natural question to ask next is whether we can improve the rate of convergence by adopting an approximation scheme of Milstein-type or more general It\^o--Taylor-type (see Chapters 5 and 10 of \cite{KloedenPlaten} for the Milstein and It\^o--Taylor schemes for classical It\^o SDEs). 
 As briefly explained in Section \ref{section_introduction}, 
\textit{the method based on the duality principle between SDEs \eqref{SDE-003-0} and \eqref{SDE-001-0} with $H\equiv 0$ does not help improve the rate of convergence.} This is because that method relies on an error estimate for the approximation of $E$, which would remain unchanged even with a higher order scheme; see Remark 3.2(4) of \cite{JumKobayashi} for details.

Below we employ the approach used in the previous section to establish the rate of strong convergence of an It\^o--Taylor-type scheme. 
This is done by directly estimating the error for the approximation of $X$, thereby avoiding the separate analysis of the error for the approximation of $E$.
However, we first discuss a Milstein-type scheme (which is a special case of the It\^o--Taylor-type scheme) in the hope that the simplest case gives the reader a clear understanding of the complicated notations and discussions of the more general case. Since the arguments we give in this section are similar to the ones given in the proof of Theorem \ref{Mil0}, we only provide a sketch for the proof of each theorem. We also continue to use the notations $\E_B,\,\E_D$ and $\E$ introduced in the beginning of Section \ref{section_3}.

In addition to Assumptions $\H_1$ and $\H_2$, assume that there exists a continuous, nondecreasing function $k:[0,\infty)\rightarrow(0,\infty)$ such that for all $x,y\in\mathbb{R}$ and $u\ge 0$, 
\begin{itemize}
\item[$\H_5$:] $\bullet$ \  $F,\ G,\ GG_x\in C^{1,2}$;
\item[\phantom{$\H_5$:}] $\bullet$ \  $|(GG_x)(u, x)-(GG_x)(u, y)|\le h(u)|x-y|$;
\item[\phantom{$\H_5$:}] $\bullet$ \  $|(GG_x)(u, x)|\le h(u)(1+|x|)$; 
\item[\phantom{$\H_5$:}] $\bullet$ \  $|f_{\alpha}(u,x)|\le k(u)(1+|x|)$,
\end{itemize}  
where $h$ is the function appearing in Assumptions $\H_1$ and $\H_2$, and $f_\alpha$ reads each integrand appearing in \eqref{remainder} below. First, approximate $E$ by $E^\delta$ as in Section \ref{section_4}. Next, define a discrete-time process $(X^\delta_{\tau_n})_{n\in \{0,1,2,\ldots,N\}}$ by $X^\delta_0:=x_0$ and
\begin{align}\label{Milstein0}
	X^\delta_{\tau_{n+1}}
&:=X^\delta_{\tau_{n}}+F(n\delta,X_{\tau_n}^\delta)\delta+G(n\delta, X^\delta_{\tau_n})(B_{(n+1)\delta}-B_{n\delta})\nonumber\\
&\ \ \ +\frac{1}{2}(GG_x)(n\delta, X^\delta_{\tau_n})\big((B_{(n+1)\delta}-B_{n\delta})^2-(\tau_{n+1}-\tau_n)\big).
\end{align}
We adopt continuous interpolation as is \eqref{011} but with an additional term corresponding to the integral of $GG_x$ included. 
As in representation \eqref{010} for the Euler--Maruyama-type scheme, the Milstein-type scheme defined by \eqref{Milstein0} satisfies   
\begin{align}\label{023}
X^\delta_s
&=x_0+\int_{0}^{s}F(E_{\tau_{n_r}},X_{\tau_{n_r}}^\delta)\ud E_r+\int_{0}^{s}G(E_{\tau_{n_r}}, X^\delta_{\tau_{n_r}})\ud B_{E_r}\notag\\
&\ \ \ +\int_0^s\int_{\tau_{n_{r_2}}}^{r_2}(GG_x)(E_{\tau_{n_{r_2}}}, X^\delta_{\tau_{n_{r_2}}})\ud B_{E_{r_1}}\ud B_{E_{r_2}}. 
\end{align}
\begin{theorem}
\label{Milstein} 
Let $X$ be the solution of SDE \eqref{SDE0} with $H\equiv 0$ such that Assumptions $\H_1$, $\H_2$ and $\H_5$ hold.
Assume that the Laplace exponent $\psi$ of $D$ is regularly varying at $\infty$ with index $\beta\in(0,1)$ and that one of the following conditions holds: 
\begin{enumerate}[(a)]
\item $h$ and $k$ are constant and $\beta\in(1/2,1)$;
\item $h$ and $k$ are continuous and nondecreasing, $h$ and $\log k$ are regularly varying at $\infty$ with indices $q\ge 0$ and $\tilde{q}\ge 0$, respectively, and $\beta\in \left((q_{\ast\ast}-1)/q_{\ast\ast},1\right)$, where $q_{\ast\ast}:=\max\{2q+2,\tilde{q}\}$. 
\end{enumerate}
Let $X^\delta$ be the approximation process of Milstein-type defined in \eqref{Milstein0} with continuous interpolation. 
Then there exists a constant $C>0$ not depending on $\delta$ such that for all $\delta\in(0,1)$,
$
	\E\left[\sup_{0\leq s\leq t}|X_s-X^\delta_s|\right]\leq C\delta.
$
Thus, $X^\delta$ converges strongly to $X$ uniformly on $[0,T]$ with order $1$.  
\end{theorem}

\begin{proof}
By SDE \eqref{SDE0} with $H\equiv 0$ and the It\^o formula, 
\begin{align}
X_{\tau_{n+1}}
&=X_{\tau_{n}}+\int_{\tau_n}^{\tau_{n+1}} F(E_{\tau_n}, X_{\tau_n})\ud E_r+\int_{\tau_n}^{\tau_{n+1}} G(E_{\tau_n}, X_{\tau_n})\ud B_{E_r}\notag\\
&\ \ \ \ \ \ \ \ \ +\int_{\tau_n}^{\tau_{n+1}}\int_{\tau_n}^{r_2}(GG_x)(E_{\tau_n}, X_{\tau_n})\ud B_{E_{r_1}}\ud B_{E_{r_2}}+R_{(\tau_n,\tau_{n+1})};\notag
\end{align}
\begin{align}
R_{(a,b)}
&:=\int_{a}^{b}\int_{a}^{r_2}\!\!\left(F_u+F_xF+\frac{1}{2}F_{xx}G^2\right)\ud E_{r_1}\ud E_{r_2}+\int_{a}^{b}\int_{a}^{r_2}\!\!F_xG\ud B_{E_{r_1}}\ud E_{r_2}\notag\\
&\ \ +\int_{a}^{b}\int_{a}^{r_2}\!\!\left(G_t+G_xF+\frac{1}{2}G^2G_{xx}\right)\!\ud E_{r_1}\ud B_{E_{r_2}}\notag\\
&\ \ +\int_{a}^{b}\int_{a}^{r_3}\!\!\int_{a}^{r_2}\!\!\left((G_xG)_t+(G_xG)_xF+\frac{1}{2}(G_xG)_{xx}G^2\right)\!\ud E_{r_1}\ud B_{E_{r_2}}\ud B_{E_{r_3}}\notag\\
&\ \ +\int_{a}^{b}\int_{a}^{r_3}\!\!\int_{a}^{r_2}\!\!(G_xG)_xG\ud B_{E_{r_1}}\ud B_{E_{r_2}}\ud B_{E_{r_3}},\label{remainder}
\end{align}
with all the integrands for $R_{(a,b)}$ evaluated at $(E_{r_1}, X_{r_1})$.
From this representation together with representation \eqref{023} for $X^\delta$, it follows that $Z_t:=\sup_{0\leq s\leq t}|X_s-X^\delta_s|\leq I_1+I_2+I_3+I_4$, where 
\begin{align*}
I_1&:=
\sup_{0\leq s\leq t}\left|\int_0^s \{F(E_{\tau_{n_r}}, X_{\tau_{n_r}})-F(E_{\tau_{n_r}}, X_{\tau_{n_r}}^\delta)\}\ud E_r\right|;\\
I_2&:=\sup_{0\leq s\leq t}\left|\int_0^s \{G(E_{\tau_{n_r}}, X_{\tau_{n_r}})-G(E_{\tau_{n_r}}, X_{\tau_{n_r}}^\delta)\}\ud B_{E_r}\right|;\\
I_3&:=\sup_{0\leq s\leq t}\left|\int_0^s\!\!\int_{\tau_{n_{r_2}}}^{r_2}\!\! \bigl\{GG_x(E_{\tau_{n_{r_2}}}, X_{\tau_{n_{r_2}}})-GG_x(E_{\tau_{n_{r_2}}}, X_{\tau_{n_{r_2}}}^\delta)\bigr\}\ud B_{E_{r_1}}\ud B_{E_{r_2}}\right|;\\
I_4&:=\sup_{0\leq s\leq t}\left|\sum_{i=0}^{n_s-1}R_{(\tau_i,\tau_{i+1})}+R_{(\tau_{n_s},s)}\right|.
\end{align*}
Arguments as in the proofs of Theorems \ref{EU} and \ref{Mil0} yield 
\begin{align*}
&\E_B[I_1^2]
\leq E_Th^2(E_T)\int_0^t\E_B[Z_r^2]\ud E_r;\ \ \  
\E_B[I_2^2]
\leq b_2h^2(E_T)\int_0^t\E_B[Z_r^2]\ud E_r;\\
&\E_B[I_3^2]\leq b_2\delta h^2(E_T)\int_0^t\E_B[Z_r^2]\ud E_r; 
\ \ \  \E_B[I_4^2]\leq \xi_5(E_T)\E_B[Y_T^{(2)}]\delta^2,
\end{align*}
where $\xi_5(u):=c(u+u^2)k^2(u)$ with some constant $c>0$ and $Y_T^{(2)}$ is defined in Proposition \ref{Yp}. (To estimate $\E_B[I_4^2]$, recall the methods used for obtaining \eqref{t2} and \eqref{0003}.)

Thus,
$
\E_B[Z_t^2]\leq 4 \xi_5(E_T)\E_B[Y_T^{(2)}]\delta^2+ 4h^2(E_T)(E_T+2b_2)\int_0^t\E_B[Z_r^2]\ud E_r, 
$
which implies 
$
\E_B[Z_T^2]\leq 4 \xi_5(E_T)\E_B[Y_T^{(2)}]e^{4h^2(E_T)(E_T+2b_2)E_T}\delta^2. 
$
The desired result now follows as in the last part of the proof of Theorem \ref{Mil0}.
\end{proof}

To discuss a general It\^o--Taylor-type scheme, we recall Chapters 5 and 9 of \cite{KloedenPlaten} and introduce shorthand notations for the following operators:
\[
L^0:=\frac{\p}{\p u}+F\frac{\p}{\p x}+\frac{1}{2}G^2\frac{\p^2}{\p x^2}; \ \ \ L^1:=G\frac{\p}{\p x}.
\]
For a multi-index $\alpha=(j_1,\dots,j_\ell)$ with $j_i\in\{0,1\}$ for $i=1,\dots, \ell$ and $\ell\ge 1$, let 
\[
f_\alpha:=L^{j_1}\cdots L^{j_{l-1}}G^{j_l}\ \, \mbox{if}\ \, \ell\ge 2\ \ \ \mbox{and}\ \ \  f_\alpha:=G^{j_1} \ \, \mbox{if}\ \, \ell=1,
\]
where  $G^0:=F$ and $G^1:=G$. Define the multiple integral of the function $f_\alpha$ by
\[
	I_\alpha(f_\alpha(E_{\cdot},X_{\cdot}))_{a,b}:=\int_{a\leq r_1\leq\cdots\leq r_\l\leq b}f_\alpha(E_{r_1},X_{r_1})\ud B^{(j_1)}_{E_{r_1}}\cdots\ud B^{(j_\ell)}_{E_{r_\ell}}
\]
with $\ud B^{(0)}_{E_r}:=\ud E_r$ and $\ud B^{(1)}_{E_r}:=\ud B_{E_r}.$
Let $v$ denote the multi-index with length $\ell=0$, and let $f_v(u,x)=x$ and $I_v(f)_{a,b}=f(a)$. 
Also, for each $\gamma$ such that $2\gamma$ is a positive integer (so that $\gamma=0.5, 1, 1.5, 2, \ldots$), let
\[
	\mathcal{A}_\gamma:=\left\{\alpha: \ell(\alpha)+n(\alpha)\leq 2\gamma \ \mbox{or}\ \ell(\alpha)=n(\alpha)=\gamma+\frac{1}{2}\right\},
\]
where $\ell(\alpha)$ denotes the length of $\alpha$ and $n(\alpha)$ denotes the number of the zero components of $\alpha$. 

Define a discrete-time process $(X^\delta_{\tau_n})_{n\in \{0,1,2,\ldots,N\}}$ by $X^{\delta }_0:=x_0$ and 
\begin{align}\label{ItoTaylor}
X^{\delta }_{\tau _{n+1}}:=X^{\delta }_{\tau _{n}}+\sum_{\alpha \in \mathcal A_{\gamma}\backslash\{v\}}I_\alpha(f_\alpha(E_{\tau_n}, X^\delta_{\tau_n}))_{\tau_n,\tau_{n+1}},
\end{align}
with continuous interpolation as in the Milstein-type scheme but with some additional terms included. 
For the solution $X$ of SDE \eqref{SDE0} with $H\equiv 0$, a repeated use of the It\^o formula yields
\begin{align*}
X_{t}&=x_0+\sum_{\alpha \in \mathcal A_{\gamma}\backslash\{v\}}\left(\sum_{i=0}^{n_t-1}I_\alpha(f_\alpha(E_{\tau_{i}}, X_{\tau_{i}}))_{\tau_{i},\tau_{i+1}}+I_\alpha(f_\alpha(E_{\tau_{n_t}}, X_{\tau_{n_t}}))_{\tau_{n_t},t}\right)\\
&\ \ \ +\sum_{\alpha \in \mathcal{R}(\mathcal{A}_\gamma)}\left(\sum_{i=0}^{n_t-1}I_\alpha(f_\alpha(E_{\cdot}, X_{\cdot}))_{\tau_{i},\tau_{i+1}}+I_\alpha(f_\alpha(E_{\cdot}, X_{\cdot}))_{\tau_{n_t},t}\right),
\end{align*}
where $\mathcal{R}(\mathcal{A}_\gamma)$ is the remainder set of multi-indices given by
\[
	\mathcal{R}(\mathcal{A}_\gamma):=\{\alpha\notin\mathcal{A}_\gamma,\ \mbox{and}\ -\alpha:=(j_2,\cdots,j_{\ell})\in \mathcal{A}_\gamma\}.
\]
We also let $\alpha -:=(j_1,\dots, j_{\ell-1})$.

Assume that there exist two families of continuous, nondecreasing functions $\{h_\alpha:[0,\infty)\rightarrow [0,\infty):\alpha\in\mathcal{A}_\gamma\backslash v\}$ and $\{k_\alpha:[0,\infty)\rightarrow (0,\infty):\alpha\in\mathcal{R}(\mathcal{A}_\gamma)\}$  such that for all $u\in[0,T]$ and $x,y\in\R$, 
\begin{enumerate}
\item[$\H_6$:] $\bullet$ \ $f_{-\alpha}\in C^{1,2}$ for $\alpha\in\mathcal{A}_\gamma\cup\mathcal{R}(\mathcal{A}_\gamma)\backslash v$;
\item[\phantom{$\H_6$:}] $\bullet$ \ 
$|f_\alpha(u,x)-f_\alpha(u,y)|\leq h_\alpha(u)|x-y|$ for $\alpha\in\mathcal{A}_\gamma\backslash v$; 
\item[\phantom{$\H_6$:}] $\bullet$ \ $|f_\alpha(u,x)|\leq h_\alpha(u)(1+|x|)$ for $\alpha\in\mathcal{A}_\gamma\backslash v$; 
\item[\phantom{$\H_6$:}] $\bullet$ \ $|f_\alpha(u,x)|\leq k_\alpha(u)(1+|x|)$ for $\alpha\in\mathcal{R}(\mathcal{A}_\gamma)$.
\end{enumerate}

\begin{theorem}\label{HighEU}
Let $X$ be the solution of SDE \eqref{SDE0} with $H\equiv 0$ such that Assumption $\H_6$ holds with $h_\alpha$ and $\log k_\alpha$ being regularly varying at $\infty$ 
with indices $q_\alpha\ge 0$ and $\tilde q_\alpha\ge 0$, respectively. 
Assume further that the Laplace exponent $\psi$ of $D$ is regularly varying at $\infty$ with index $\beta\in(0,1)$.
Let $X^\delta$ be the approximation process of It\^o--Taylor-type defined in \eqref{ItoTaylor} with continuous interpolation. 
If $\beta\in \left((q_{\ast\ast\ast}-1)/q_{\ast\ast\ast},1\right)$, where 
$
	q_{***}:=\max\left\{
		\max_{\alpha\in\mathcal{A}_\gamma\backslash v} (2q_\alpha+2),\,
		\max_{\alpha\in \mathcal{R}(\mathcal{A}_\gamma)}\tilde q_{\alpha}
		\right\},
$
then there exists a constant $C>0$ not depending on $\delta$ such that for all $\delta\in(0,1)$,
$
 	\E\left[\sup_{0\leq s\leq t}|X_s-X^\delta_s|\right]\leq C\delta^{\gamma}.
$
Thus, $X^\delta$ converges strongly to $X$ uniformly on $[0,T]$ with order $\gamma$.  Moreover, if both $h_\alpha$ and $k_\alpha$ are constant for all $\alpha\in\mathcal{A}_\gamma\backslash v$ and $\alpha\in\mathcal{R}(\mathcal{A}_\gamma)$, respectively, the same conclusion holds as long as $\beta\in(1/2,1)$.
\end{theorem}

Our proof for this theorem relies on the following two lemmas that help determine a bound for an integral involving $f_\alpha$ with $\alpha\in \mathcal{R}(\mathcal{A}_\gamma)$ in the way in which $\E_B[I_4]$ in the proof of Theorem \ref{Mil0} was estimated. The two lemmas can be regarded as generalizations of Lemmas 5.7.3 and 10.8.1 in \cite{KloedenPlaten}, respectively, and we omit the proofs since they are similar to the ones given in \cite{KloedenPlaten}. 
Note that we will only need the second lemma in order to prove Theorem \ref{HighEU}, but we list the first lemma as well since it plays an important role in proving the second lemma. In the statements of the lemmas, we assume $f$ is a given function for which all multiple integrals appearing in the proofs are well-defined. 

\begin{lemma} \label{single} For any multi-index $\alpha$ and $r\ge 0$, 
\[
\E_B\bigg[\!\sup_{\tau_{n_r}\!\leq\! \sigma\leq r}\!\!I_{\alpha}(f(E_{\cdot},X_{\cdot}))_{\tau_{n_r},\sigma}^2\bigg]\leq \E_B\bigg[\!\sup_{\tau_{n_r}\leq \sigma\leq r}\!\!|f(E_\sigma,X_\sigma)|^2\bigg]b_2^{\l(\alpha)-n(\alpha)}\delta^{\l(\alpha)+n(\alpha)}.
\]
\end{lemma}
\begin{lemma}\label{multi} For any multi-index $\alpha$ and $t\ge 0$, 
\begin{align*}
&\E_B\left[\sup_{0\leq s\leq t}\left|\sum_{i=0}^{{n_s}-1}I_\alpha(f(E_\cdot, X_\cdot))_{\tau_i,\tau_{i+1}}+I_\alpha(f(E_\cdot, X_\cdot))_{\tau_{n_s},s}\right|^2\right]\\
&\leq(4b_2^{\ell(\alpha)-n(\alpha)+1}+E_t)\delta^{\phi(\alpha)}\int_0^t\E_B\left[\sup_{0\leq \sigma\leq r}|f(E_\sigma,X_\sigma)|^2\right]\ud E_r,
\end{align*}
where $\phi(\alpha):=\l(\alpha)+n(\alpha)-1$ if $\l(\alpha)\ne n(\alpha)$, and $\phi(\alpha):=2\l(\alpha)-2$ if $\l(\alpha)= n(\alpha)$. 
\end{lemma}

\noindent\textit{Proof of Theorem \ref{HighEU}}\ \ 
Let $Z_t^2:=\sup_{0\leq s\leq t}|X_s-X^\delta_s|^2$. 
By the Cauchy-Schwartz inequality, 
\begin{align*}
Z_t^2&\leq 2|\mathcal{A}_\gamma|\sum_{\alpha\in\mathcal{A}_\gamma\backslash v}\sup_{0\leq s\leq t}\Bigg|\sum_{i=0}^{n_s-1}
I_\alpha\left(f_\alpha(E_{\tau_i}, X_{\tau_i})-f_\alpha(E_{\tau_i},X^\delta_{\tau_i})\right)_{\tau_i,\tau_{i+1}}\\
&\ \ \ \ \ \ \ \ \ \ \ \ \ \ \ \ \ \ \ \ \ \ \ \ \ \ \ \ \ \ \ \ \ +I_\alpha\left(f_\alpha(E_{\tau_{n_s}}, X_{\tau_{n_s}})-f_\alpha(E_{\tau_{n_s}},X^\delta_{\tau_{n_s}})\right)_{\tau_{n_s},s}\Bigg|^2\\
& +2|\mathcal{R}(\mathcal{A}_\gamma)|\!\!\sum_{\alpha\in\mathcal{R}(\mathcal{A}_\gamma)}\!\sup_{0\leq s\leq t}\Bigg|\!\sum_{i=0}^{n_s-1}
I_\alpha\left(f_\alpha(E_\cdot, X_\cdot)\right)_{\tau_i,\tau_{i+1}}\!+\!I_\alpha\left(f_\alpha(E_\cdot, X_\cdot)\right)_{\tau_{n_s},s}\!\Bigg|^2.
\end{align*}
By Assumption $\H_6$, $\sup_{0\leq \sigma\leq r} |f_\alpha(E_\sigma,X_\sigma)-f_\alpha(E_\sigma, X^\delta_\sigma)|^2\leq h_\alpha^2(E_r)Z_r^2$ for all $\alpha \in\mathcal{A}_\gamma\backslash v$ and $\sup_{0\leq \sigma\leq r} |f_\alpha(E_\sigma,X_\sigma)|^2\leq 2k_\alpha^2(E_r)Y_r^{(2)}$ for all $\alpha\in\mathcal{R}(\mathcal{A}_\gamma)$. 
By an argument analogous to the one given 
for obtaining \eqref{t2} and \eqref{0003}, together with Lemma \ref{multi} and the inequality $\ell(\alpha)-n(\alpha)+1\le 2\gamma+2$ valid for any $\alpha\in \mathcal{A}_\gamma\cup \mathcal{R}(\mathcal{A}_\gamma)$,
\begin{align*}
\E_B[Z_t^2]&\leq2(4b_2^{2\gamma+2}+E_T)|\mathcal{A}_\gamma|\sum_{\alpha\in\mathcal{A}_\gamma\backslash v} h_\alpha^2(E_t)\delta^{\phi(\alpha)}\int_0^t\E_B[Z_r^2]\ud E_r\\
&\ \ \ +4E_T(4b_2^{2\gamma+2}+E_T)\E_B[Y_T^{(2)}]|\mathcal{R}(\mathcal{A}_\gamma)|\sum_{\alpha\in\mathcal{R}(\mathcal{A}_\gamma)} k_\alpha^2(E_t)\delta^{\phi(\alpha)},
\end{align*}
where $Y_T^{(2)}$ is defined in Proposition \ref{Yp}.
Note that the least value of $\phi(\alpha)$ for $\alpha\in\mathcal{R}(\mathcal{A}_\gamma)$ is $2\gamma$, which is achieved when $\l(\alpha -)+n(\alpha -)=2\gamma$, $j_\ell=1$ and $\l(\alpha)\ne n(\alpha)$. 
Thus, by the Gronwall-type inequality, 
\begin{align*}
\E_B[Z_T^2]&\leq cE_T(4b_2^{2\gamma+2}+ E_T)\E_B[Y_T^{(2)}]\xi_6(E_T)e^{c\xi_7(E_T)(2b_2+2+ E_T)E_T}\delta^{2\gamma},
\end{align*}
where $c>0$ is a constant,  
$
	\xi_6(u):=\sum_{\alpha\in\mathcal{R}(\mathcal{A}_\gamma)} k_\alpha^2(u)=\sum_{\alpha\in\mathcal{R}(\mathcal{A}_\gamma)} e^{2\log k_\alpha(u)},
$
and 
$
	\xi_7(u):=\sum_{\alpha\in\mathcal{A}_\gamma\backslash v} h_\alpha^2(u).
$	
The desired results now follow as in the last paragraph of the proof of Theorem \ref{Mil0}.\qed 

\section{Simulations}\label{section_6}
This section presents two numerical examples which verify the rates of convergence obtained in Theorems \ref{EU} and \ref{Mil0} in Section \ref{section_4}, respectively.
We use a $0.8$-stable subordinator $D$ and simulate a sample path of $X^\delta$ on a time interval $[0,T]$ as follows.  
First, simulate $D$ at the discretization points $\{0,\delta, 2\delta,\cdots\}$ and stop this procedure upon finding an integer $N$ satisfying $T\in[D_{N\delta}, D_{(N+1)\delta})$. 
Second, simulate the Brownian motion $B$ at the discretization points $\{0,\delta, 2\delta,\cdots, N\delta\}$.
Finally, calculate $X^{\delta}_{n\delta}$ for $n=0,1,2,\cdots,N-1$ using the approximation scheme in \eqref{def-EM} and let $X^{\delta}_{t}:=X^{\delta}_{n\delta}$ whenever $t\in[D_{n\delta},D_{(n+1)\delta})$.

Consider the two SDEs 
\begin{align}
X_t&=1+\int_0^t\sqrt{1+E_r}\,\ud r+\int_0^t\sqrt{1+E_r}X_r\,\ud E_r+\int_0^t\sqrt{1+E_r}X_r\,\ud B_{E_r}\label{Ex1};\\
Y_t&=1+\int_0^t\sqrt{1+E_r}\,\ud r+\int_0^t\sqrt{1+E_r}Y_r\,\ud E_r+\int_0^t\sqrt{1+E_r}\,\ud B_{E_r}.\label{Ex2}
\end{align}
It is easy to verify that the coefficients of SDEs \eqref{Ex1} and \eqref{Ex2} satisfy the conditions of Theorems \ref{EU} and \ref{Mil0} with $\theta=1$, respectively. 
To examine the order of convergence for SDE \eqref{Ex1}, 
as in \cite{DengLiu2020}, we regard the numerical solution with the fixed step size $\delta_0=2^{-15}$ as the true solution. 
For each fixed $\delta\in \{2^{-14},\,2^{-13},\dots,\,2^{-8}\}$, 
we generate $100$ paths for the true solution $X^{\delta_0}$ and 100 paths for the approximated solution $X^\delta$. We then calculate the following error at the time horizon $T=1$:
\begin{align*}
ERROR(X,\delta)&:=\frac{1}{100}\sum_{i=1}^{100}|X^{\delta}_T(\omega_i)-X^{\delta_0}_T(\omega_i)|.
\end{align*}

\begin{figure}
\centering
\vspace{-140pt}
    \includegraphics[width=3.9in]{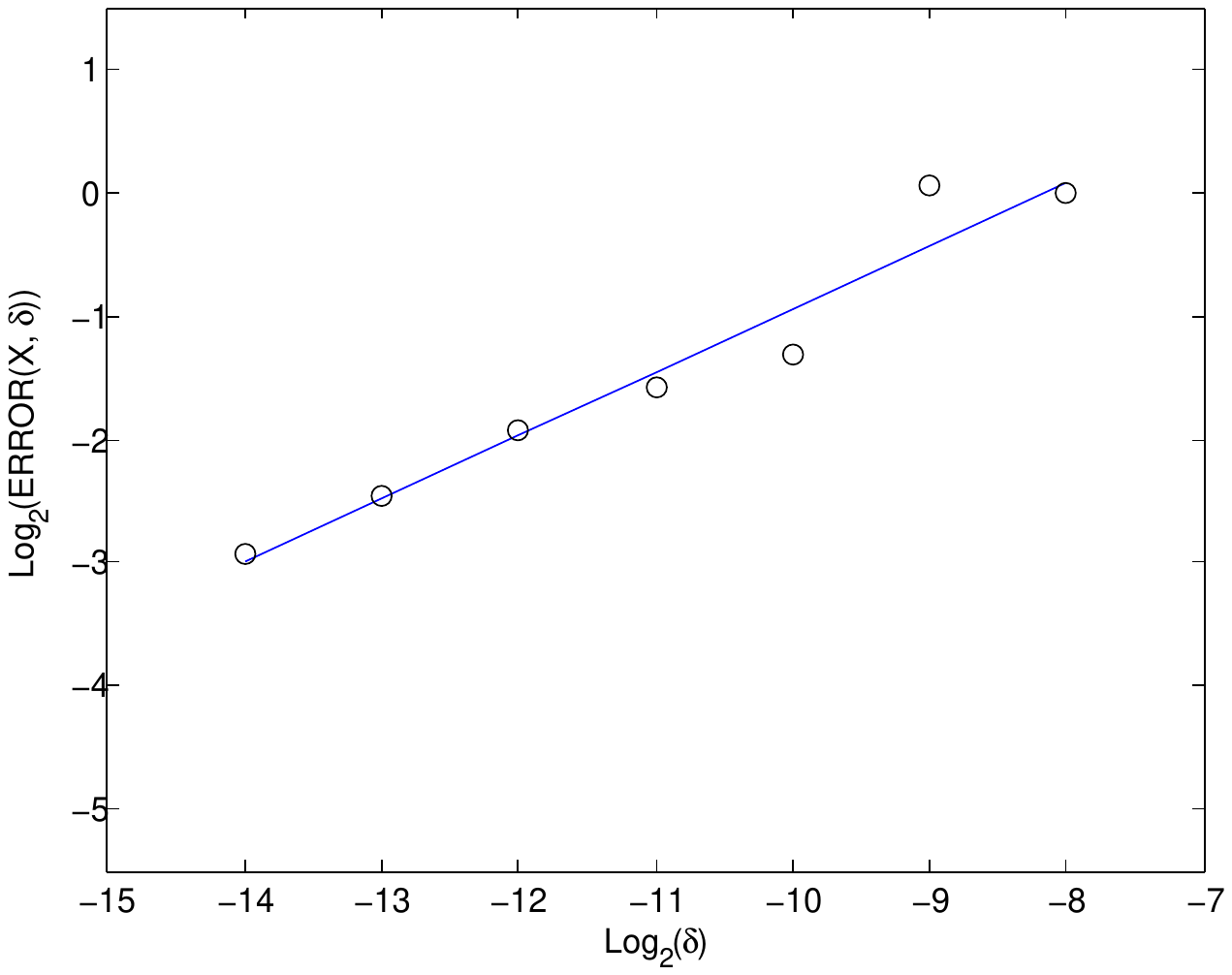}
\vspace{-120pt}
    \caption{Plot of $\log_{2}(ERROR(X,\delta))$ versus $\log_{2}\delta$ with the least squares line $y= 0.5138x+4.1982$.}
    \label{F1}
\vspace{-110pt}
 \includegraphics[width=3.9in]{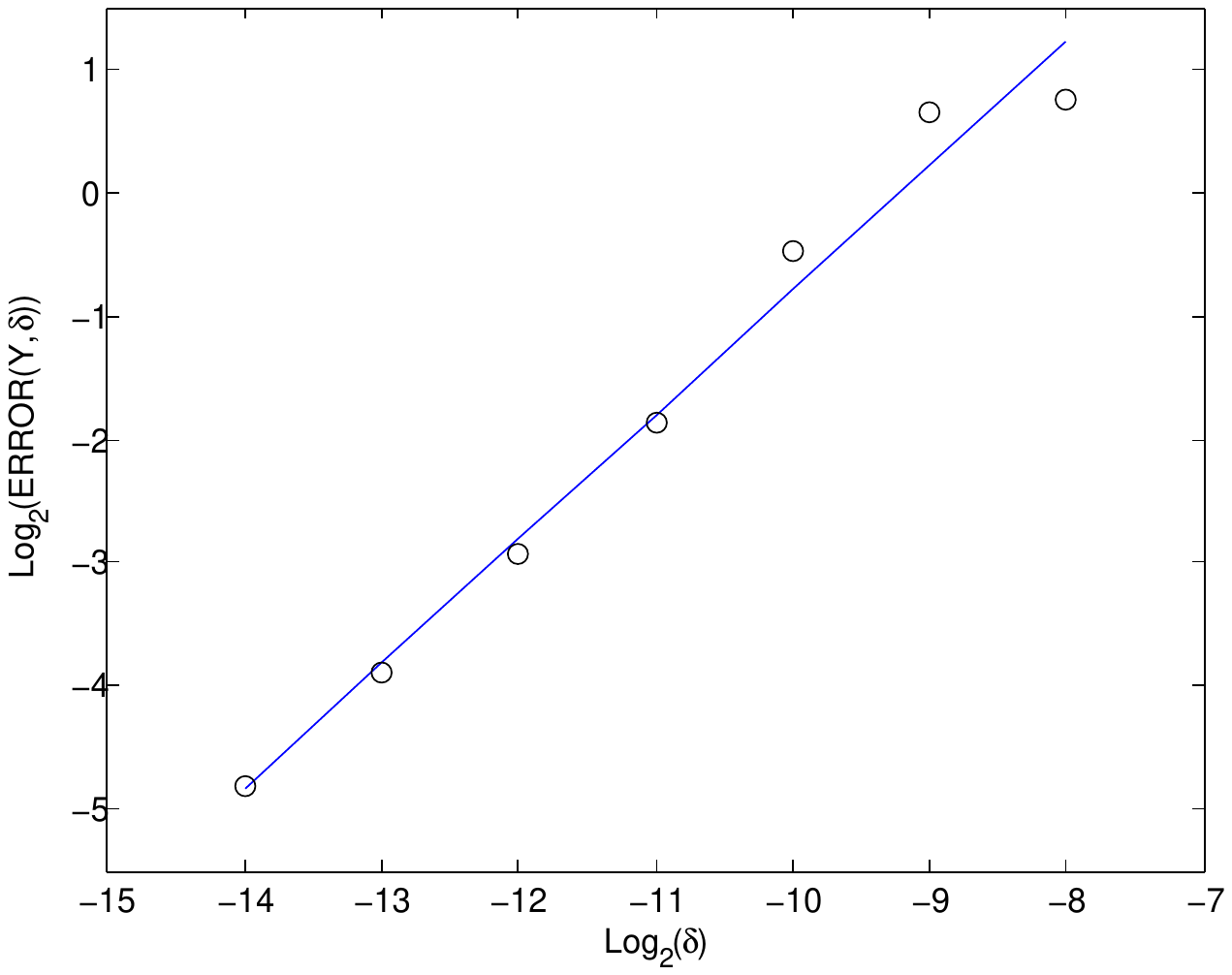}
\vspace{-120pt}
   \caption{Plot of $\log_{2}(ERROR(Y,\delta))$ versus $\log_{2}\delta$ with the least squares line $y= 1.0088x+ 9.2991$.}
    \label{F2}
\end{figure}

Figure \ref{F1} gives a plot of $\log_{2}(ERROR(X,\delta))$ against $\log_{2}(\delta)$. It shows a linear trend with least squares slope being 0.5138, which is close to 0.5, the largest possible slope suggested by Theorem \ref{EU}. 
On the other hand, for SDE \eqref{Ex2}, Figure \ref{F2} gives a plot of $\log_{2}(ERROR(Y,\delta))$ against $\log_{2}(\delta)$, which again presents a linear trend but with least squares slope being 1.0088. The latter is close to 1 as suggested by Theorem \ref{Mil0}.

\bigskip
\textbf{Acknowledgements:}
The authors are very grateful to the anonymous referees and the associate editor for their comments and suggestions that helped improve the exposition of the paper and make some statements precise.

\end{document}